\documentclass[UTF8,fontset=windows,final, leqno]{siamltex}
\usepackage{ctex}
\usepackage{amsmath}
\usepackage{epsfig}
\usepackage{graphicx}
\usepackage{amssymb}
\usepackage{CJK}
\usepackage{color}

\usepackage{hyperref}
\usepackage[noadjust]{cite}

\usepackage{caption}

\allowdisplaybreaks[4]

\numberwithin{equation}{section}
\newtheorem{assumption}{Assumption}[section]
\newtheorem{algorithm}{Algorithm}[section]
\newtheorem{remark}{Remark}[section]

\renewcommand\abstractname{Abstract}

\renewcommand\figurename{Figure}
\renewcommand\tablename{Table}

\renewcommand\refname{References}

\def\lam{{\lambda}}

\def\Ome{{\Omega}}

\def\nab{{\nabla}}
\def\vepsi{{\varepsilon}}
\def\p{{\partial}}
\def\reff#1{\eqref{#1}}
\def\norm#1#2{\Vert\,#1\,\Vert_{#2}}

\def\vepsi{\varepsilon}

\def\cT{{\mathcal T}}

\def\no{{\nonumber}}

\def\div{{\mbox{\rm div\,}}}

\def\p{{\partial}}

\def\nab{\nabla}
\def\Ome{\Omega}
\def\lam{\lambda}

\newcommand{\bRM}{\mathbf{RM}}
\newcommand{\br}{\mathbf{r}}

\def\ba{\mathbf{a}}
\def\bb{\mathbf{b}}

\def\bC{\mathbf{C}}
\def\bbf{\mathbf{f}}
\def\bu{\mathbf{u}}

\def\bv{\mathbf{v}}
\def\bw{\mathbf{w}}

\def\bg{\mathbf{g}}
\def\bn{\mathbf{n}}
\def\bH{\mathbf{H}}
\def\bV{\mathbf{V}}
\def\bL{\mathbf{L}}
\def\bP{\mathbf{P}}

\def\bV{\mathbf{V}}

\def\bX{\mathbf{X}}

\def\R{\mathbb{R}}

\CTEXoptions[today=old]

\begin{document}
	
	
	\title{Optimal error estimates of multiphysics finite element method for a nonlinear poroelasticity model with nonlinear stress-strain relations
		\footnote{Last update: \today}}
	
	\author{
		Zhihao Ge\thanks{School of Mathematics and Statistics, Henan University, Kaifeng 475004, P.R. China ({\tt Email:zhihaoge@henu.edu.cn}).
			The work of this author was supported by the National Natural Science Foundation of China under grant No.11971150.}
		\and
		Hairun Li\thanks{School of Mathematics and Statistics, Henan University, Kaifeng 475004, P.R. China.}
       \and
		Tingting Li\thanks{School of Mathematics and Statistics, Henan University, Kaifeng 475004, P.R. China ({\tt Email:ltt1120@mail.ustc.edu.cn}).
			The work of this author was supported by the National Natural Science Foundation of China under grant  No. 11801143.}
		%
	}
	
	\maketitle
	
	
	\setcounter{page}{1}
	
	
	
	\begin{abstract}
		In this paper, we study the numerical algorithm for a nonlinear poroelasticity model with nonlinear stress-strain relations. By using variable substitution, the original problem can be reformulated to a new coupled fluid-fluid system, that is, a generalized nonlinear Stokes problem of displacement vector field related to pseudo pressure and a diffusion problem of other pseudo pressure fields. A new technique is used to get the existence and uniqueness of the solution of the reformulated model and a fully discrete nonlinear finite element method is proposed to solve the model numerically. The multiphysics finite element is used to get the discretization of the space variable and the backward Euler method is taken as the time-stepping method in the fully discrete case. Stability analysis and the error estimation are given for the fully discrete case and numerical test are taken to verify the theoretical results.
	\end{abstract}
	
	\begin{keywords}
		Nonlinear poroelasticity model, Multiphysics finite element method, Backward Euler method.
	\end{keywords}
	
	\begin{AMS}
		65M12, 
		65M15, 
		65N30. 
	\end{AMS}
	
	\pagestyle{myheadings}
	\thispagestyle{plain}
	\markboth{ZHIHAO GE, HAIRUN LI, TINGTING LI}{OPTIMAL ERROR ESTIMATES OF MFEM FOR NONLINEAR POROELASTICITY}
	

\section{Introduction}\label{sec-1}

Poromechanic is a fluid-solid interaction system at pore scale and is a branch of continuum mechanics and acoustics. If the solid is an elastic material, then the subject of the study is known as poroelasticity. Moreover, the elastic material may be governed by linear or nonlinear constitutive law, which then leads respectively to linear and nonlinear poroelasticity. The quasi-static poroelasticity model to be studied in this paper is given as  (cf. \cite{pw07,20210819,X.B.Feng2014}):
\begin{alignat}{2} \label{e1.1}
	-\div \sigma(\bu) + \alpha \nab p &= \bbf
	&&\qquad \mbox{in } \Ome_T:=\Ome\times (0,T)\subset \mathbf{\R}^d\times (0,T),\\
	(c_0p+\alpha \div \bu)_t + \div \bv_f &=\phi &&\qquad \mbox{in } \Ome_T,
	\label{e1.2}
\end{alignat}
where
\begin{align} \label{e1.3}
	\bv_f &:= -\frac{K}{\mu_f} \bigl(\nab p -\rho_f \bg \bigr).
\end{align}
Here $\Omega \subset \R^d \,(d=1,2,3)$ denotes a bounded polygonal domain with the boundary
$\p\Ome$. And  $\bu$ denotes the displacement vector of the solid and
$p$ denotes the pressure of the solvent. $\bbf$ is the body force, $\bg$ is the gravitational acceleration. $I$ denotes the $d\times d$ identity matrix. The  permeability tensor $K=K(x)$  is assumed to be symmetric and uniformly positive definite in the sense that there exists positive constants $K_1$ and $K_2$ such that $K_1|\zeta|^2\leq K(x)\zeta\cdot \zeta \leq K_2 |\zeta|^2$ for a.e. $x\in\Omega$ and any $\zeta\in \mathbf{\R}^d$. The solvent viscosity $\mu_f$, Biot-Willis constant $\alpha$, and the constrained specific storage coefficient $c_0$ are all positive.  In addition, $\sigma(\bu)$ is called the nonlinear (effective) stress tensor,
$\widehat{\sigma}(\bu, p):=\sigma(\bu)-\alpha p I$ is the total stress tensor.  $\bv_f$ is the volumetric solvent flux and \reff{e1.3} is the well-known Darcy's law, and $\rho_f$ is a positive constant. 

To close the system  (\ref{e1.1})-(\ref{e1.2}), the following set of boundary and initial conditions will be considered in this paper:
\begin{alignat}{2} \label{e1.4a}
	\widehat{\sigma}(\bu,p)\bn=\sigma(\bu)\bn-\alpha p \bn &= \bbf_1
	&&\qquad \mbox{on } \p\Ome_T:=\p\Ome\times (0,T),\\
	\bv_f\cdot\bn= -\frac{K}{\mu_f} \bigl(\nab p -\rho_f \bg \bigr)\cdot \bn
	&=\phi_1 &&\qquad \mbox{on } \p\Ome_T, \label{e1.4b} \\
	\bu=\bu_0,\qquad p&=p_0 &&\qquad \mbox{in } \Ome\times\{t=0\}. \label{e1.4c}
\end{alignat}

We remark that poroelasticity model (\ref{e1.1})-(\ref{e1.2}) has been widely applied in many fields such as material science, biomechanics and so on, for the details, one can see \cite{T.Roose2003,C.C.Swan2003,20210819} and the references therein.  Especially, the nonlinear poroelasticity model has been used to simulate the activity process of some biological tissues, for example, nonlinear poroelasticity model can be used to study soft tissue of orbit and mechanical mechanism of lung and other organs in human body, one can refer to \cite{borgerlung,luboz} and the references therein. Due to the complexity of nonlinear poroelasticity, there are few of theoretical results, as for the numerical results for the nonlinear poroelasticity problem, Levenston and Frank proposed a three-field (displacement, fluid flux, pressure) formula in \cite{LevenstonM}, Chapelle and Gerbeau et al. proposed the operator splitting iteration method in \cite{ChapelleD}, Berger, Bordas and Kay developed a stabilized low-order three-field mixed finite element method in \cite{BergerL} for the fully saturated incompressible small deformation case for which a linearly elastic model is sufficient. The low-order finite element method is relatively easy to implement and allows effective preconditioning by \cite{FerronatoM, HughesTJ, WhiteJA}. Vuong, Yoshihara and Wall use a three-field finite element method which is continuous in pressure in \cite{VuongAT}. Berger, Bordas, Kay and Tavener construct a stabilized finite-element method to compute flow and finite-strain deformations in an incompressible poroelastic medium in \cite{borger}, where the authors employ a three-field mixed formulation to calculate displacement, fluid flux and pressure directly and introduce a Lagrange multiplier to enforce
flux boundary conditions. The authors of \cite{pw07,murad}  use the finite element method to solve the poroelasticity model directly and it will appear locking phenomenon. Recently, a HHO method is proposed for nonlinear poroelasticity and elasticity in \cite{BOFFI2019} and \cite{DANIELEBOFFI2016}. As for the permeability tensor $K(\div \bu)$ of nonlinear poroelasticity, one can refer to  \cite{20210921,20210922,20210923,20210924,20210925,20210927,20210928} and the references therein. 
 At the same time, another main difficulty is how to decouple the coupled system and how to deal with the nonlinear terms in computation. In this work, we follow the method used in \cite{X.B.Feng2014} to get the numerical algorithm. By introducing a new variable $q={\rm \div}\mathbf{u}$ and a new denotation $\mathcal{N}(\varepsilon(\mathbf{u}))=\sigma(\mathbf{u})-\lambda \mathrm{div}\mathbf{u}\textit{I}$, where Lam\'e constant $\lam$ is computed from the  Young's modulus $E$ and the  Poisson ratio $\nu$ by $\lam=\frac{E\nu}{(1+\nu)(1-2\nu)}$, we reformulate the nonlinear poroelasticity problem into a generalized nonlinear Stokes problem for the displacement vector field and pseudo pressure field and a diffusion problem for other pseudo pressure fields. To the best of our knowledge, there are few existing references to use the fully discrete multiphysics finite element method to study the nonlinear poroelasticity. In this work, the multiphysics nonlinear finite element method is used to solve the space variables, and the Newton iteration method is used to solve the nonlinear equations, and the backward Euler method is used to get the time discretization. The main novelty is that the multiphysics nonlinear finite element method is used and the error estimate is made for the numerical method.

This paper is organized as follows. In Section \ref{sec-2}, we introduce the reformulated multiphysics model for the nonlinear poroelasticity and give the PDE analysis. In Section \ref{sec-3}, we propose a fully discrete multiphysics nonlinear finite element method for the poroelasticity model, and give the stability analysis and convergence analysis of the multiphysics nonlinear finite element method. In Section \ref{sec-4}, we present the numerical results to verify the efficiency of the proposed numerical methods. Finally, we draw a conclusion to summary the main works in this paper.

\section{Multiphysics reformulation and its PDE analysis}\label{sec-2}
\subsection{Multiphysics reformulation} \label{sec-2.2}
To reveal the multi-physics progresses and to propose absolutely stable and high order numerical method, we firstly derive a multiphysics reformulation for the original model, and then make numerical approximation for the reformulated model. This is a key idea of this work and it will be seen in the later sections that this new approach is advantageous. To the end, we introduce new variable $q:=\div \bu$, and denote
\[\mathcal{N}(\varepsilon(\mathbf{u})):~=\sigma(\mathbf{u})-\lambda q\textit{I}, \quad\eta:=c_0p+\alpha q,\quad \xi:=\alpha p -\lam q,
\]
where $
\varepsilon(\mathbf{u})=\frac{1}{2}(\nabla\mathbf{u}+\nabla\mathbf{u}^{T})$.

It is easy to check that
\begin{align}\label{e1.5}
	p=\kappa_1 \xi + \kappa_2 \eta, \qquad q=\kappa_1 \eta-\kappa_3 \xi,
\end{align}
where
\begin{align}\label{e1.6}
	\kappa_1:= \frac{\alpha}{\alpha^2+\lam c_0},
	\quad \kappa_2:=\frac{\lam}{\alpha^2+\lam c_0}, \quad
	\kappa_3:=\frac{c_0}{\alpha^2+\lam c_0}.
\end{align}

Then \reff{e1.1}-\reff{e1.2} can be written as
\begin{alignat}{2} \label{e1.7}
	-\div\mathcal{N}(\varepsilon(\bu)) + \nab \xi &= \bbf &&\qquad \mbox{in } \Ome_T,\\
	\kappa_3\xi +\div \bu &=\kappa_1\eta &&\qquad \mbox{in } \Ome_T, \label{e1.8}\\
	\eta_t - \frac{1}{\mu_f} \div[K (\nab (\kappa_1 \xi + \kappa_2 \eta)-\rho_f\bg)]&=\phi
	&&\qquad \mbox{in } \Ome_T, \label{e1.9}
\end{alignat}
where $p$ and $q$ are related to $\xi$ and $\eta$
through the algebraic equations in \reff{e1.5}.

The boundary and initial conditions \eqref{e1.4a}-\eqref{e1.4c} can be rewritten as
\begin{alignat}{2} \label{20150712_1}
	\sigma(\bu)\bn-\alpha (\kappa_1 \xi + \kappa_2 \eta) \bn &= \bbf_1
	&&\qquad \mbox{on } \p\Ome_T,\\
	-\frac{K}{\mu_f} \bigl(\nab (\kappa_1 \xi + \kappa_2 \eta) -\rho_f \bg \bigr)\cdot \bn
	&=\phi_1 &&\qquad \mbox{on } \p\Ome_T, \label{20150712_2} \\
	\bu=\bu_0,\qquad p&=p_0 &&\qquad \mbox{in } \Ome\times\{t=0\}. \label{20150712_3}
\end{alignat}

\begin{remark}\label{rem210122-1}
It is clear that $(\bu, \xi)$ satisfies a generalized nonlinear Stokes problem for
a given $\eta$, where $\kappa_3\xi$ in \eqref{e1.8} acts as a penalty term, and $\eta$ satisfies a diffusion problem for a given $\xi$. This new formulation reveals the underlying deformation and diffusion  processes which occurs in the poroelastic material.
\end{remark}
\begin{remark}\label{rem211007-1}
The original pressure is eliminated in the reformulation, which will be helpful to overcome the ``locking phenomenon", the later numerical tests show that our proposed method has no ``locking phenomenon" ( see Section \ref{sec-4}).
\end{remark}

To define the weak solution, we need the standard function space notation, one can see \cite{bs08,cia,temam} for details.
In particular, $(\cdot,\cdot)$ and $\langle \cdot,\cdot\rangle$
denote respectively the standard $L^2(\Ome)$ and $L^2(\p\Ome)$ inner products.
For any Banach space $B$, we let $\mathbf{B}=[B]^d$,
and use $\mathbf{B}^\prime$ to denote its dual space. In particular,
we use $(\cdot,\cdot)$ 
to denote the dual product on $\bH^1(\Ome)' \times \bH^1(\Ome)$,
and $\norm{\cdot}{L^p(B)}$ is a shorthand notation for
$\norm{\cdot}{L^p((0,T);B)}$.


Letting $A=(a_{ij})_{n\times n}$, we respectively define the product $(A, A)$, $F$-norm $\|A\|_F$ and the $L^2$-norm $\|A\|_{L^2(\Omega)}$ of the matrix by
\begin{eqnarray}
	\|A\|_{F} =:(A:A)^\frac{1}{2}=: (\sum_{i,j=1}^{n}a_{ij}^2)^\frac{1}{2},\label{202002181}\\
	(A, A)=\int_\Omega A:A dx=:\|A\|_{L^2(\Omega)} ({\rm or\ denoted\ by }\ \|A\|_{L^2}  ).\label{eq20210119-1}
\end{eqnarray}

We also introduce the function spaces
\begin{align*}
	&L^2_0(\Omega):=\{q\in L^2(\Omega);\, (q,1)=0\}, \qquad \bX:= \bH^1(\Ome).
\end{align*}
From \cite{temam},  it is well known that the following so-called
inf-sup condition holds in the space $\bX\times L^2_0(\Ome)$:
\begin{align}\label{e2.0}
	\sup_{\bv\in \bX}\frac{(\div \bv,\varphi)}{\norm{\bv}{H^1(\Ome)}}
	\geq \alpha_0 \norm{\varphi}{L^2(\Ome)} \qquad \forall
	\varphi\in L^2_0(\Ome),\quad \alpha_0>0.
\end{align}

Let
\[
\bRM:=\{\br:=\ba+\bb \times x;\, \ba, \bb, x\in \R^d\}
\]
denotes the space of infinitesimal rigid motions. From \cite{brenner,gra,temam}, it is well known  that $\bRM$ is the kernel of
the strain operator $\vepsi$, that is, $\br\in \bRM$ if and only if
$\vepsi(\br)=0$. Hence, we have
\begin{align}
	\vepsi(\br)=0,\quad \div \br=0 \qquad\forall \br\in \bRM. \label{e4.100}
\end{align}

$\bL^2_\bot(\p\Ome)$ and $\bH^1_\bot(\Ome)$ denote respectively the
subspaces of $\bL^2(\p\Ome)$ and $\bH^1(\Ome)$ which are orthogonal to $\bRM$, that is,
\begin{align*}
	&\bH^1_\bot(\Ome):=\{\bv\in \bH^1(\Ome);\, (\bv,\br)=0\,\,\forall \br\in \bRM\},
	\\
	&\bL^2_\bot(\p\Ome):=\{\bg\in \bL^2(\p\Ome);\,\langle \bg,\br\rangle=0\,\,
	\forall \br\in \bRM \}.
\end{align*}


Next, we define the weak solutions to problem \reff{e1.1}--\reff{e1.4c}. For convenience, we assume that $\bbf, \bbf_1,\phi$
and $\phi_1$ all are independent of $t$ in the remaining of the paper. Note that all the results of this paper can be easily extended to the case of time-dependent source functions.

\begin{definition}\label{weak1}
	Let $\bu_0\in\bH^1(\Ome), \bbf\in\bL^2(\Omega),
	\bbf_1\in \bL^2(\p\Ome), p_0\in L^2(\Ome), \phi\in L^2(\Ome)$,
	and $\phi_1\in  L^2(\p\Ome)$.  Assume $c_0>0$ and
	$(\bbf,\bv)+\langle \bbf_1, \bv \rangle =0$ for any $\bv\in \mathbf{RM}$.
	Given $T > 0$, a tuple $(\bu,p)$ with
	\begin{alignat*}{2}
		&\bu\in L^\infty\bigl(0,T; \bH_\perp^1(\Ome)),
		&&\qquad p\in L^\infty(0,T; L^2(\Omega))\cap L^2 \bigl(0,T; H^1(\Omega)\bigr), \\
		&p_t, (\div\bu)_t \in L^2(0,T;H^{1}(\Ome)'),
		&&\qquad 
	\end{alignat*}
	is called a weak solution to the problem \reff{e1.1}-\reff{e1.4c},
	if for almost every $t \in [0,T]$ there holds
	\begin{alignat}{2}\label{e2.1}
		&\bigl( \mathcal{N}(\varepsilon(\bu)), \vepsi(\bv) \bigr)
		+\lam\bigl(\div\bu, \div\bv \bigr)
		-\alpha \bigl( p, \div \bv \bigr)  && \\
		&\hskip 2in
		=(\bbf, \bv)+\langle \bbf_1,\bv\rangle
		&&\quad\forall \bv\in \bH^1(\Ome), \no \\
		&\bigl((c_0 p +\alpha\div\bu)_t, \varphi \bigr)
		+ \frac{1}{\mu_f} \bigl( K(\nab p-\rho_f\bg), \nab \varphi \bigr)
		\label{e2.2} \\
		&\hskip 2in =\bigl(\phi,\varphi\bigr)
		+\langle \phi_1,\varphi \rangle
		&&\quad\forall \varphi \in H^1(\Ome), \no  \\
		&\bu(0) = \bu_0,\qquad p(0)=p_0.  && \label{e2.3}
	\end{alignat}
\end{definition}

\medskip
Similarly, we can define a weak solution to the problem \reff{e1.7}-\reff{e1.9}, \reff{20150712_1}-\reff{20150712_3}.

\begin{definition}\label{weak2}
	Let $\bu_0\in \bH^1(\Ome), \bbf \in \bL^2(\Omega),
	\bbf_1 \in \bL^2(\p\Ome), p_0\in L^2(\Ome), \phi\in L^2(\Ome)$,
	and $\phi_1\in L^2(\p\Ome)$.  Assume $c_0>0$ and
	$(\bbf,\bv)+\langle \bbf_1, \bv \rangle =0$ for any $\bv\in \mathbf{RM}$.
	Given $T > 0$, a $3$-tuple $(\bu,\xi,\eta)$ with
	\begin{alignat*}{2}
		&\bu\in L^\infty\bigl(0,T; \bH_\perp^1(\Ome)), &&\quad
		\xi\in L^\infty \bigl(0,T; L^2(\Omega)\bigr), \\
		&\eta\in L^\infty\bigl(0,T; L^2(\Omega)\bigr)
		\cap H^1\bigl(0,T; H^{1}(\Omega)'\bigr),
		&&
	\end{alignat*}
	is called a weak solution to the problem \reff{e1.7}-\reff{e1.9}, \reff{20150712_1}-\reff{20150712_3}
	if for almost every $t \in [0,T]$, there has
	\begin{alignat}{2}\label{e2.4}
		\bigl(\mathcal{N}(\varepsilon(\bu)), \vepsi(\bv) \bigr)-\bigl( \xi, \div \bv \bigr)
		&= (\bbf, \bv)+\langle \bbf_1,\bv\rangle
		&&\quad\forall \bv\in \bH^1(\Ome), \\
		\kappa_3 \bigl( \xi, \varphi \bigr) +\bigl(\div\bu, \varphi \bigr)
		&= \kappa_1\bigl(\eta, \varphi \bigr) &&\quad\forall \varphi \in L^2(\Ome), \label{e2.5}  \\
		\bigl(\eta_t, \psi \bigr)
		+\frac{1}{\mu_f} \bigl(K(\nab (\kappa_1\xi +\kappa_2\eta) &-\rho_f\bg), \nab \psi \bigr) \label{e2.6} \\
		&= (\phi, \psi)+\langle \phi_1,\psi\rangle &&\quad\forall \psi \in H^1(\Ome) , \no  \\
		p:=\kappa_1\xi +\kappa_2\eta, \qquad
		&q:=\kappa_1\eta-\kappa_3\xi, && \label{e2.7} \\
		\eta(0)= \eta_0:&=c_0p_0+\alpha q_0,  && \label{e2.9}
	\end{alignat}
	where $q_0:=\div \bu_0$, $u_0$ and $p_0$ are same as in Definition \ref{weak1}.
\end{definition}


\begin{remark}\label{rem-2.1}
	The reason for introducing the space $\bH_\perp^1(\Ome)$
	in the above two definitions is that the boundary condition \eqref{e1.4a}
	is a pure "Neumann condition". If it is replaced by a pure Dirichlet condition or by a mixed Dirichlet-Neumann condition, there is no need to introduce this space.  Thus, from the analysis point of view, the pure Neumann condition case is the most difficult case.
\end{remark}

\subsection{Existence, uniqueness and regularity of the weak solution}\label{sec-2.3}

From \cite{dautray}, we know that there exists a constant $c_1>0$ such that
\[
\inf_{\br\in \bRM}\|\bv+\br\|_{L^2(\Ome)}
\le c_1\|\vepsi(\bv)\|_{L^2(\Ome)} \qquad\forall \bv\in\bH^1(\Ome).
\]
Hence, for each $\bv\in \bH^1_\bot(\Ome)$ there holds
\begin{align}\label{e4.1+}
	\|\bv\|_{L^2(\Ome)}=\inf_{\br\in \bRM} \sqrt{\|\bv+\br\|_{L^2(\Ome)}^2-\|\br\|_{L^2(\Ome)}^2 }
	\le c_1\|\vepsi(\bv)\|_{L^2(\Ome)},
\end{align}
and by the well-known Korn's inequality in \cite{dautray}, there exists some $c_2>0$ satisfies that
\begin{align} \label{e4.1+0}
	\|\bv\|_{H^1(\Ome)} &\le c_2[\|\bv\|_{L^2(\Ome)}+\|\vepsi(\bv)\|_{L^2(\Ome)}]\\
	&\le c_2(1+c_1)\|\vepsi(\bv)\|_{L^2(\Ome)} \qquad\forall \bv\in\bH^1_\bot(\Ome).\no
\end{align}

By Lemma 2.1 of \cite{brenner} we know that for any $q\in L^2(\Ome)$, there
exists $\bv\in \bH^1_\bot(\Ome)$ such that $\div \bv=q$ and
$\|\bv\|_{H^1(\Ome)} \leq C\|q\|_{L^2(\Ome)}$. An immediate consequence of
this lemma is that there holds the following alternative version of the inf-sup condition:
\begin{align}\label{e2.0a}
	\sup_{\bv\in \bH^1_\bot(\Ome)}\frac{(\div \bv,\varphi)}{\norm{\bv}{H^1(\Ome)}}
	\geq \alpha_1 \norm{\varphi}{L^2(\Ome)} \qquad \forall
	\varphi\in L^2_0(\Ome),\quad \alpha_1>0.
\end{align}
To obtain the existence, uniqueness and regularity of the weak solution, we need the following assumption:

\begin{assumption}\label{assu210204-1}
 There exist $C_{c}, C_{l}, C_{m}>0$ and $C_{c}, C_{m}>\sqrt{2}\lambda$ such that
\begin{eqnarray}
	(\sigma(\mathbf{u}),\varepsilon(\mathbf{u})) &\geq& C_{c}\|\varepsilon(\mathbf{u})\|_{L^{2}}^{2}, \label{202006028}\\
	\|\sigma(\mathbf{u})-\sigma(\mathbf{v})\| &\leq& C_{l}\|\varepsilon(\mathbf{u})-\varepsilon(\mathbf{v})\|_{L^{2}}, \label{202006029}\\
	(\sigma(\mathbf{u})-\sigma(\mathbf{v}),\varepsilon(\mathbf{u})-\varepsilon(\mathbf{v}))
	&\geq& C_{m}\|\varepsilon(\mathbf{u})-\varepsilon(\mathbf{v})\|_{L^{2}}^{2}. \label{2020060210}
\end{eqnarray}
\end{assumption}
We remark that Assumption \ref{assu210204-1} is reasonable for that there exist some terms of $\sigma(\mathbf{u})$ satisfies \reff{202006028}-\reff{2020060210}, as for the details, one can refer to \cite{BOFFI2019} or \cite{mddp}, and we omit the details here.

\begin{lemma}\label{weak3}
	Assume that\ $\mathbf{u}, \mathbf{v}\in L^\infty\bigl(0,T; H_\perp^1(\Omega))$, there exists $C_{cv}, C_{lp}, C_{mn}>0$ satisfy the following conditions:
	\begin{eqnarray}
		(\mathcal{N}(\varepsilon(\mathbf{u})),\varepsilon(\mathbf{u})) &\geq& C_{cv}\|\varepsilon(\mathbf{u})\|_{L^{2}}^{2}, \label{202002032}\\
		\|\mathcal{N}(\varepsilon(\mathbf{u}))-\mathcal{N}(\varepsilon(\mathbf{v}))\| &\leq& C_{lp}\|\varepsilon(\mathbf{u})-\varepsilon(\mathbf{v})\|_{L^{2}}, \label{202002033}\\
		(\mathcal{N}(\varepsilon(\mathbf{u}))-\mathcal{N}(\varepsilon(\mathbf{v})),\varepsilon(\mathbf{u})-\varepsilon(\mathbf{v}))
		&\geq& C_{mn}\|\varepsilon(\mathbf{u})-\varepsilon(\mathbf{v})\|_{L^{2}}^{2}. \label{202002034}
	\end{eqnarray}
\end{lemma}
\begin{proof} Using \reff{202006029}, we have
		\begin{eqnarray}
		&&\|\mathcal{N}(\varepsilon(\mathbf{u}))-\mathcal{N}(\varepsilon(\mathbf{v}))\|_{L^{2}}=\|\sigma(\mathbf{u})-\lambda \div\mathbf{u}I-\sigma(\mathbf{v})+\lambda \div\mathbf{v}I\|_{L^{2}}\nonumber\\
		&&\leq \|\sigma(\mathbf{u})-\sigma(\mathbf{v})\|_{L^{2}}+\lambda\|\div\mathbf{u}I-\div\mathbf{v}I\|_{L^{2}} \nonumber \\
		&&\leq C_{l}\|\varepsilon(\mathbf{u})-\varepsilon(\mathbf{v})\|_{L^{2}}+\lambda\|\div\mathbf{u}I-\div\mathbf{v}I\|_{L^{2}}\nonumber\\
		&&\leq C_{l}\|\varepsilon(\mathbf{u})-\varepsilon(\mathbf{v})\|_{L^{2}}+\sqrt{2}\lambda\|\varepsilon(\mathbf{u})-\varepsilon(\mathbf{v})\|_{L^{2}}\nonumber\\
		&&=(C_{l}+\sqrt{2}\lambda)\|\varepsilon(\mathbf{u})-\varepsilon(\mathbf{v})\|_{L^{2}},\no
	\end{eqnarray}
which implies that \reff{202002033} holds with $C_{lp}:=C_{l}+\sqrt{2}\lambda$.
	
Similarly, using \reff{202002032} and the Cauchy-Schwartz inequality, we get
	\begin{eqnarray}\label{202003172}
		&&(\mathcal{N}(\varepsilon(\mathbf{u})),\varepsilon(\mathbf{u}))=( \sigma(\mathbf{u})-\lambda \div\mathbf{u}I,\varepsilon(\mathbf{u}) )\\
		&& \geq C_{c}\|\varepsilon(\mathbf{u})\|_{L^{2}}^{2}-\lambda(\div\mathbf{u}I,\varepsilon(\mathbf{u})) \geq C_{c}\|\varepsilon(\mathbf{u})\|_{L^{2}}^{2}-\sqrt{2}\lambda \|\varepsilon(\mathbf{u})\|_{L^{2}}^{2}\nonumber \\
		&&= (C_{c}-\sqrt{2}\lambda) \|\varepsilon(\mathbf{u})\|_{L^{2}}^{2},\nonumber
	\end{eqnarray}
	which implies that \eqref{202002032} holds with $C_{cv}:=C_{c}-\sqrt{2}\lambda>0$.
	
It is easy to check that \eqref{202002034} holds if $C_{mn}:=C_{m}-\sqrt{2}\lambda>0$. The proof is complete.
\end{proof}

\begin{lemma}\label{estimates}
There exists a positive constant
$ C_1=C_1\bigl(\|\bu_0\|_{H^1(\Ome)}, \|p_0\|_{L^2(\Ome)},$
$\|\bbf\|_{L^2(\Ome)},\|\bbf_1\|_{L^2(\p \Ome)},\|\phi\|_{L^2(\Ome)}, \|\phi_1\|_{L^2(\p\Ome)} \bigr)$
such that
\begin{align}\label{e2.15d}
&\sqrt{C_{cv}}\|\varepsilon(\bu)\|_{L^\infty(0,T;L^2(\Ome))}
+\sqrt{\kappa_2} \|\eta\|_{L^\infty(0,T;L^2(\Ome))} \\
&\qquad
+\sqrt{\kappa_3} \|\xi\|_{L^\infty(0,T;L^2(\Ome))}
+\sqrt{\frac{K_1}{\mu_f}} \|\nab p \|_{L^2(0,T;L^2(\Ome))} \leq C_1, \no \\
&\|\bu\|_{L^\infty(0,T;L^2(\Ome))}\leq C_1, \quad
\|p\|_{L^\infty(0,T;L^2(\Ome))} \leq C_1 \bigl( \kappa_2^{\frac12} + \kappa_1 \kappa_3^{-\frac12}
\bigr), \label{e2.15e} \\
&\|p\|_{L^2(0,T; L^2(\Ome))} \leq C_1, \quad
\|\xi\|_{L^2(0,T;L^2(\Ome))} \leq C_1\kappa_1^{-1} \bigl(1+ \kappa_2^{\frac12} \bigr).
\label{e2.15f}
\end{align}
\end{lemma}
\begin{proof}
Using \eqref{e2.6} and taking $\psi=p$, we have
\begin{align}\label{1}
 \int_0^t\bigl(\eta_t, p \bigr)ds
&+\int_0^t\frac{1}{\mu_f} \bigl(K(\nab (\kappa_1\xi +\kappa_2\eta) -\rho_f\bg), \nab p \bigr)ds\\
&= \int_0^t[(\phi, p)+\langle \phi_1,p\rangle]ds \quad\forall \psi \in H^1(\Ome).\no
\end{align}
Taking $v=u_{t}$ in \eqref{e2.4} and differential \eqref{e2.5} with respect to $t$ and takes $\varphi=\xi$, we have
\begin{align}\label{2}
 \int_0^t\bigl(\eta_t, p \bigr)ds &=\int_0^t\bigl(\eta_t, \kappa_1\xi +\kappa_2\eta \bigr)ds \\
   &=\int_0^t\bigl(\eta_t, \kappa_1\xi\bigr)ds +\int_0^t\bigl(\eta_t, \kappa_2\eta \bigr)ds\no\\
   &=\frac{1}{2}\kappa_2 (\norm{\eta(t)}{L^2(\Ome)}^2-\norm{\eta(0)}{L^2(\Ome)}^2)+ \int_0^t\bigl(\eta_t, \kappa_1\xi\bigr)ds\no
\end{align}
and
\begin{align}\label{3}
 \int_0^t\bigl(\eta_t, \kappa_1\xi\bigr)ds&= \int_0^t\bigl(q_{t}+\kappa_3\xi_{t}, \xi\bigr)ds \\
   &= \int_0^t\bigl(\div\bu_{t}, \xi\bigr)ds+\frac{1}{2}\kappa_3\bigl( \norm{\xi(t)}{L^2(\Ome)}^2-\norm{\xi(0)}{L^2(\Ome)}^2\bigr)\no\\
   &=\int_0^t[\bigl(\mathcal{N}(\varepsilon(\bu)), \vepsi(\bu_{t}) \bigr)- (\bbf, \bu_{t})+\langle \bbf_1,\bu_{t}\rangle]ds\no\\
   &\qquad\qquad+\frac{1}{2}\kappa_3\bigl( \norm{\xi(t)}{L^2(\Ome)}^2-\norm{\xi(0)}{L^2(\Ome)}^2\bigr).\no
\end{align}
Substituting \eqref{2} and \eqref{3} into \eqref{1}, we obtain
\begin{align}
&\int_0^t\bigl(\mathcal{N}(\varepsilon(\bu)), \vepsi(\bu_{t}) \bigr)ds + \frac12 \Bigl[\kappa_2 \norm{\eta(t)}{L^2(\Ome)}^2 +\kappa_3 \norm{\xi(t)}{L^2(\Ome)}^2] \label{eq210121-2}\\
&\qquad\qquad+\frac{1}{\mu_f} \int_0^t \bigl(K(\nab p-\rho_f\bg), \nab p\bigr)\, ds\nonumber\\
&= \frac12 \Bigl[\kappa_2 \norm{\eta(0)}{L^2(\Ome)}^2 +\kappa_3 \norm{\xi(0)}{L^2(\Ome)}^2 -2\bigl(\bbf,\bu(0)\bigr) -2\langle \bbf_1, \bu(0) \rangle \Bigr]\nonumber\\
&\qquad\qquad+\int_0^t \bigl(\phi, p\bigr)\, ds+\int_0^t \langle \phi_1, p \rangle\, ds.\nonumber
\end{align}

Using \eqref{202002032} and \reff{eq210121-2}, and Poincare inequality, we see that  \eqref{e2.15d} and \eqref{e2.15e}  holds.

Note that \eqref{e2.15f} follows from \eqref{e2.15d}, \eqref{e4.1+}, the Poincar\`e
inequality, and the relation $p=\kappa_1\xi +\kappa_2\eta$.
\end{proof}


\begin{lemma}\label{smooth}
Let  $(\bu, \xi, \eta)$ be the weak solution of the problem \reff{e1.7}-\reff{e1.9} with \reff{e1.4a}-\reff{e1.4c}, and suppose that $\bu_0$ and $p_0$ are sufficiently smooth, then there exist positive constants $C_2=C_2\bigl(C_1,\|\nab p_0\|_{L^2(\Ome)} \bigr)$
and $C_3=C_3\bigl(C_1,C_2, \|\bu_0\|_{H^2(\Ome)},\|p_0\|_{H^2(\Ome)} \bigr)$ such that
\begin{align}\label{eq3.21}
&\sqrt{C_{cv}}\|\varepsilon(\bu_{t})\|_{L^2(\Ome)}
+\sqrt{\kappa_2} \|\eta_t\|_{L^2(0,T;L^2(\Ome))} \\
&\qquad
+\sqrt{\kappa_3} \|\xi_t\|_{L^2(0,T;L^2(\Ome))}
+\sqrt{\frac{K_1}{\mu_f}} \|\nab p \|_{L^\infty(0,T;L^2(\Ome))} \leq C_2, \no \\
&\sqrt{C_{cv}}\|\varepsilon(\bu_{t})\|_{L^2(\Ome)}
+\sqrt{\kappa_2} \|\eta_t\|_{L^\infty(0,T;L^2(\Ome)) \label{eq3.22}} \\
&\qquad
+\sqrt{\kappa_3} \|\xi_t\|_{L^\infty(0,T;L^2(\Ome))}
+\sqrt{\frac{K_1}{\mu_f}} \|\nab p_t \|_{L^2(0,T;L^2(\Ome))} \leq C_3, \no \\
&\|\eta_{tt}\|_{L^2(H^{1}(\Ome)')} \leq \sqrt{\frac{K_2}{\mu_f}}C_3. \label{eq3.23}
\end{align}
\end{lemma}
\begin{proof}
Note that $\bbf, \bbf_1,\phi$ and $\phi_1$ all are assumed to be independent of $t$. differentiating \eqref{e2.4} and \eqref{e2.5} with respect to $t$, taking $\bv=\bu_t$
and $\varphi=\xi_t$ in \eqref{e2.4} and \eqref{e2.5} respectively, and adding the
resulting equations up, we have
\begin{align}\label{eq3.23a}
\bigl(\mathcal{N}_{t}(\varepsilon(\bu)),\varepsilon(\bu_{t})\bigr) = \bigl(q_t,\xi_t \bigr)
=\kappa_1 \bigl(\eta_t,\xi_t\bigr) -\kappa_3\|\xi_t\|_{L^2(\Ome)}^2.
\end{align}
Setting $\psi=p_t=\kappa_1\xi_t + \kappa_2\eta_t$ in \eqref{e2.6} gives
\begin{align}\label{eq3.23b}
\kappa_1 \bigl(\eta_t,\xi_t \bigr) + \kappa_2\|\eta_t\|_{L^2(\Ome)}^2
+\frac{K}{2\mu_f} \frac{d}{dt} \|\nab p-\rho_f\bg\|_{L^2(\Ome)}^2
=\frac{d}{dt}\Bigl[ (\phi,p) +\langle \phi_1, p\rangle \Bigr].
\end{align}
By adding \eqref{eq3.23a} and \eqref{eq3.23b} and integrating on the interval $[0,T]$, we have
\begin{align*}
&\frac{K}{2\mu_f} \|\nab p(t)-\rho_f\bg\|_{L^2(\Ome)}^2
+\int_0^t \Bigl[ \bigl(\mathcal{N}_{t}(\varepsilon(\bu)),\varepsilon(\bu_{t})\bigr)
+ \kappa_2\|\eta_t\|_{L^2(\Ome)}^2 + \kappa_3\|\xi_t\|_{L^2(\Ome)}^2\Bigr]\,ds \\
&\hskip 0.65in
=\frac{K}{2\mu_f} \|\nab p_0-\rho_f\bg\|_{L^2(\Ome)}^2
+ (\phi,p(t)-p_0) +\langle \phi_1, p(t)-p_0 \rangle, \no
\end{align*}
The above inequality and \eqref{202002032} implies \eqref{eq3.21}.

To prove \eqref{eq3.22}, firstly, differentiating \eqref{e2.4} one time with respect to $t$ and setting
$\bv=\bu_{tt}$, differentiating \eqref{e2.5} twice with respect to $t$ and setting
$\varphi=\xi_t$, and adding the resulting equations, we get
\begin{align}\label{eq3.23d}
\bigl(\mathcal{N}_{t}(\varepsilon(\bu)),\varepsilon(\bu_{tt})\bigr) = \bigl(q_{tt},\xi_t \bigr)
=\kappa_1\bigl(\eta_{tt},\xi_t\bigr) -\frac{\kappa_3}{2}\frac{d}{dt}\|\xi_t\|_{L^2(\Ome)}^2.
\end{align}
Secondly, differentiating \eqref{e2.6} with respect to $t$ one time and taking
$\psi=p_t=\kappa_1\xi_t + \kappa_2\eta_t$, we have
\begin{align}\label{eq3.23e}
\kappa_1 \bigl(\eta_{tt},\xi_t \bigr) + \frac{\kappa_2}{2} \frac{d}{dt} \|\eta_t\|_{L^2(\Ome)}^2
+\frac{K}{\mu_f} \|\nab p_t\|_{L^2(\Ome)}^2 =0.
\end{align}
Finally, adding the above two qualities and integrating on the interval $[0,t]$, we can get
\begin{align}\label{eq3.23f}
&2\int_0^t\bigl(\mathcal{N}_{t}(\varepsilon(\bu)),\varepsilon(\bu_{tt})\bigr)\,ds +\kappa_2 \|\eta_t(t)\|_{L^2(\Ome)}^2
+\kappa_3 \|\xi_t(t)\|_{L^2(\Ome)}^2 + \\
&\frac{2K}{\mu_f} \int_0^t \|\nab p_t\|_{L^2(\Ome)}^2\,ds
=\kappa_2 \|\eta_t(0)\|_{L^2(\Ome)}^2
+\kappa_3 \|\xi_t(0)\|_{L^2(\Ome)}^2.\no
\end{align}
Using \eqref{eq3.23f} and \eqref{202002032}, we can deduce that \eqref{eq3.22} holds.

Using \eqref{eq3.22} and  the following inequality
\begin{align*}
\bigl( \eta_{tt}, \psi \bigr) = -\frac{1}{\mu_f} \bigl( K\nab p, \nab \psi\bigr)
\leq \frac{K}{\mu_f} \|\nab p\|_{L^2(\Ome)} \|\nab \psi\|_{L^2(\Ome)}\qquad \forall \psi\in H^1_0(\Ome),
\end{align*}
 we implies that \eqref{eq3.23} holds. The proof is complete.
\end{proof}

\begin{theorem}\label{lem2.4}
Assume that $\bu_0\in \bH^1(\Ome), \bbf \in \bL^2(\Omega),
\bbf_1 \in \bL^2(\p\Ome), p_0\in L^2(\Ome), \phi\in L^2(\Ome)$, $\phi_1\in L^2(\p\Ome)$ and Assumption \ref{assu210204-1} hold, then there exists a unique weak solution $(\bu,\xi,\eta)$ of the problem \reff{e2.4}-\reff{e2.9}.
\end{theorem}

\begin{proof} Using Lemma \ref{estimates}, Lemma \ref{smooth} and  Schauder fixed point theorem (cf. \cite{Evans2016}), we can deduce that the solution of \reff{e2.4}-\reff{e2.9} exists.
Taking $\psi\equiv 1$ in \reff{e2.6}, we get
\begin{equation}\label{202002036}
  C_\eta(t):=\bigl(\eta(\cdot, t),1\bigr)
=\bigl(\eta_0,1\bigr) + \bigl[ (\phi,1) + \langle \phi_1, 1\rangle \bigr] t,\quad t\geq 0,
\end{equation}
which implies that  $\eta$ is unique.
From \reff{e2.4} and \reff{e2.5}, we obtain
\begin{align}
  \bigl(\mathcal{N}(\varepsilon(\bu_{1}))-\mathcal{N}(\varepsilon(\bu_{2})), \vepsi(\bv) \bigr)-\bigl( \xi_{1}-\xi_{2}, \div \bv \bigr)
&=0   &&\quad\forall \bv\in \bH^1(\Ome), \label{202002037}\\
\kappa_3 \bigl( \xi_{1}-\xi_{2}, \varphi \bigr) +\bigl(\div\bu_{1}-\div\bu_{2}, \varphi \bigr)
&= 0  &&\quad\forall \varphi \in L^2(\Ome). \label{202002038}
\end{align}
Take  $\bv=\bu_{1}-\bu_{2}, ~\varphi=\xi_{1}-\xi_{2}$ in \reff{202002037}-\reff{202002038}, and by using \reff{202002034}, we have
\begin{equation}\label{202002039}
  0\leq C_{mn}\|\vepsi(\bu_{1})-\vepsi(\bu_{2})\|^{2}_{L^{2}}+\kappa_3\|\xi_{1}-\xi_{2}\|^{2}_{L^{2}}=0,
\end{equation}
which implies that $\bu_{1}= \bu_{2}$.

From \reff{202002039} and \reff{e2.7}, we can get that $p_{1}-p_{2}=0, q_{1}-q_{2}=0$. The proof is complete.
\end{proof}

\section{Fully discrete multiphysics finite element methods}\label{sec-3}

\subsection{Formulation of fully discrete finite element methods}\label{sec-3.2}
Let $\mathcal{T}_h$ be a
quasi-uniform triangulation or rectangular partition of $\Omega$ with maximum mesh size $h$,
and $\bar{\Omega}=\bigcup_{K\in\mathcal{T}_h}\bar{K}$. Also, let $(\bX_h, M_h)$ be a stable
mixed finite element pair, that is, $\bX_h\subset \bH^1(\Omega)$ and $M_h\subset L^2(\Omega)$
satisfy the inf-sup condition
\begin{alignat}{2}\label{1407-1}
\sup_{v_h\in \bX_h}\frac{({\rm div} v_h, \varphi_h)}{\|v_h\|_{H^1(\Ome)}}
\geq \beta_0\|\varphi_h\|_{L^2(\Ome)} &&\quad \forall\varphi_h\in M_{0h}:=M_h\cap L_0^2(\Omega),\ \beta_0>0.
\end{alignat}

A number of stable mixed finite element spaces $(\bX_h, M_h)$ have been known in the literature
\cite{brezzi}.  A well-known example is the following
so-called Taylor-Hood element (cf. \cite{ber,brezzi}):
\begin{align*}
\bX_h &=\bigl\{\bv_h\in \bC^0(\overline{\Ome});\,
\bv_h|_K\in \bP_2(K)~~\forall K\in \cT_h \bigr\}, \\
M_h &=\bigl\{\varphi_h\in C^0(\overline{\Ome});\, \varphi_h|_K\in P_1(K)
~~\forall K\in \cT_h \bigr\}.
\end{align*}

Finite element approximation space $W_h$ for $\eta$ variable can be chosen
independently, any piece-wise polynomial space is acceptable provided that
$W_h \supset M_h$. 
The most convenient choice is $W_h =M_h$, which
will be adopted in the remainder of this paper.

We define
\begin{equation}\label{e3.50}
\bV_h=\bigl\{\bv_h\in \bX_h;\,  (\bv_h,\br)=0\,\,
\forall \br\in \bRM \bigr\}.
\end{equation}
It is easy to check that $\bX_h=\bV_h\bigoplus \bRM$. It was proved in \cite{fh10}
that there holds the following alternative version of the above inf-sup condition:
\begin{align}\label{e3.51}
\sup_{\bv_h\in \bV_h}\frac{(\div\bv_h,\varphi_h)}{\norm{\bv_h}{H^1(\Ome)}}
\geq \beta_1 \norm{\varphi_h}{L^2(\Ome)} \quad \forall \varphi_h\in M_{0h}, \quad \beta_1>0.
\end{align}

\begin{algorithm}\label{alg210206-1} The fully discrete  multiphysics finite element algorithm for the nonlinear poroelasticity.
\begin{itemize}
\item[(i)]
Compute $\bu^0_h\in \bV_h$ and $q^0_h\in W_h$ by $
\bu^0_h =\bu_0, p^0_h =p_0$.

\item[(ii)] For $n=0,1,2, \cdots$,  do the following two steps.
\end{itemize}

{\em Step 1:} Get $(\bu^{n+1}_h,\xi^{n+1}_h, \eta^{n+1}_h)\in \bV_h\times M_h \times  W_h$ from
\begin{eqnarray}\label{e3.1}
&&~~~~\bigl(\mathcal{N}(\varepsilon(\bu^{n+1}_h)), \vepsi(\bv_h) \bigr)-\bigl( \xi^{n+1}_h, \div \bv_h \bigr)
= (\bbf, \bv_h)+\langle \bbf_1,\bv_h\rangle \quad \forall \bv_h\in \bV_h, \\
&&~~~~\kappa_3\bigl(\xi^{n+1}_h, \varphi_h \bigr) +\bigl(\div\bu^{n+1}_h, \varphi_h \bigr)
=\kappa_1\bigl( \eta^{n+\theta}_h, \varphi_h \bigr)
~~~~~~~~~~~~~~~\quad \forall \varphi_h \in M_h, \label{e3.2} \\
\label{e3.4}
&&~~~~\bigl(d_t\eta^{n+1}_h, \psi_h \bigr)
+\frac{1}{\mu_f} \bigl(K(\nab (\kappa_1\xi^{n+1}_h +\kappa_2\eta^{n+1}_h) \\
&&\hskip 1in
-\rho_f\bg,\nab\psi_h \bigr)=(\phi, \psi_h)+\langle \phi_1,\psi_h\rangle, ~~~~~~~~~~~~~\quad  \forall \psi_h\in W_h,  \no
\end{eqnarray}
where $\theta=0$ or $1$.

{\em Step 2:} Update $p^{n+1}_h$ and $q^{n+1}_h$ by
\begin{eqnarray}
p^{n+1}_h=\kappa_1\xi^{n+1}_h +\kappa_2\eta^{n+\theta}_h, \quad
q^{n+1}_h=\kappa_1\eta^{n+1}_h-\kappa_3\xi^{n+1}_h.\label{e3.5}
\end{eqnarray}
\end{algorithm}
%

\begin{remark}\label{rem3.1}
(a) In Step 1, the nonlinear equations of \reff{e3.1}-\reff{e3.2} can be solved by the Newton iterative method as follows:
%
%
%
%
 \begin{align}\label{20201219}
    (\mathcal{N}_{h}(\varepsilon(\mathbf{u}^{n+1,i+1}_h)),\varepsilon(\mathbf{v}_{h}))-(\xi_{h}^{n+1,i+1},\mathrm{div}\bv_{h})
    =(\bbf, \bv_h)+\langle \bbf_1,\bv_h\rangle,
 \end{align}
   \begin{equation}\label{20201220}
   \kappa_3\bigl(\xi^{n+1,i+1}_h, \varphi_h \bigr) +\bigl(\div\bu^{n+1,i+1}_h, \varphi_h \bigr)
           =\kappa_1\bigl( \eta^{n+\theta,i+1}_h, \varphi_h \bigr),
   \end{equation}
%
%
%
where	$\bigl(\mathcal{N}_{h}(\varepsilon(\bu^{n+1,i+1}_h)),  \vepsi(\bv_h)\bigr)=\bigl(\mathcal{N}(\varepsilon(\bu^{n+1,i}_h))+D\mathcal{N}(\varepsilon(\bu^{n+1,i}_h))\cdot\Delta\varepsilon(\bu^{n+1,i+1}_h)), \vepsi(\bv_h) \bigr)$.


(b) If $\theta=0$,  Step 1 consists of two decoupled sub-problems which can be solved
independently. On the other hand, if $\theta=1$, Algorithm \ref{alg210206-1} is a coupled system.
\end{remark}

\subsection{Existence and uniqueness of numerical solution}\label{sec-3.3}

Taking the similar argument of \cite{BOFFI2019}, we obtain the following result of existence and uniqueness.
\begin{theorem}\label{weak11}
The numerical solution $\{({\mathbf{u}_{h}^{n}},\ {\xi_{h}^{n}},\ {\eta_{h}^{n}})\}_{n\geq0}$ of the fully discrete  multiphysics finite element algorithm \reff{e3.1}-\reff{e3.4} exists uniquely.
\end{theorem}

\begin{proof}
Denote $a_{h}(\mathbf{u}_{h}^{n},\mathbf{v}_{h})=(\mathcal{N}(\varepsilon(\mathbf{u}_{h}^{n})),\varepsilon(\mathbf{v}_{h}))$, also define the following bilinear form
$a_{h}^{lin}(\mathbf{u}_{h}^{n},\mathbf{v}_{h})$:
\begin{equation}
	a_{h}^{lin}(\mathbf{u}_{h}^{n},\mathbf{v}_{h})=C_{cv}(\varepsilon(\mathbf{u}_{h}^{n}), \varepsilon(\mathbf{v}_{h})), \forall \mathbf{v}_{h}\in V_{h}.\label{eq210205-1}
\end{equation}

We consider the following linear auxiliary problem: find $(y_{h}^{n+1},\ {\xi_{h}^{n+1}},\ {\eta_{h}^{n+\theta}})\in V_{h}\times M_{h}\times M_{h}$ satisfy that
\begin{eqnarray}
&& a_{h}^{lin}(\mathbf{y}_{h}^{n+1},\mathbf{v}_{h})
-( \xi^{n+1}_h, {\rm div} \mathbf{v}_h )
 =({\textbf{f}},{\textbf{v}_{h}})
 +\langle{\textbf{f}_{1}},{\textbf{v}_h}\rangle,\quad\forall{\textbf{v}_{h}}\in{V_{h}},\nonumber\\
 &&\kappa_{3}(\xi_{h}^{n+1},{\varphi_{h}})
 +({\rm div} \mathbf{y}_{h}^{n+1},{\varphi_{h}})=
k_{1}(\eta_{h}^{n+\theta},{\varphi_{h}}),\quad\forall \varphi_{h}\in M_{h},\label{202003041}\\
&&({d_{t}\eta_{h}^{n+1}},\psi_{h})+\frac{1}{\mu_{f}}({K(\nabla(\kappa_{1}\xi_{h}^{n+1} +\kappa_{2}\eta_{h}^{n+1})-{\rho_{f}}\mathbf{g})},{\nabla\psi_{h}})\nonumber\\
&&=(\phi,\psi_{h})+\langle{\phi_{1}},\psi_{h}\rangle,\quad\forall \psi_{h}\in M_{h}\nonumber
\end{eqnarray}
with the boundary conditions of \eqref{20150712_1}- \eqref{20150712_3}.

Following the method of \cite{X.B.Feng2014}, we can prove that the solution of \reff{202003041} uniquely exists, here we omit the details of the proof.

Next, we define a mapping~$\Phi_{h}: V_{h}\rightarrow V_{h}$ such that
\begin{equation}\label{202003044}
  a_{h}^{lin}(\Phi_{h}(\mathbf{v}_{h}^{n+1}),\mathbf{w}_{h})= a_{h}(\mathbf{v}_{h}^{n+1},\mathbf{w}_{h}),\quad\forall \mathbf{w}_{h}\in V_{h}.
\end{equation}

Next, we will show that $\Phi_{h}$ is an isomorphism by two steps:

(i) Prove that $\Phi_{h}$ is a injective mapping. Assume that $\mathbf{v}_{h,1}^{n+1}\neq\mathbf{v}_{h,2}^{n+1}\in V_{h}$ satisfy $\Phi_{h}(\mathbf{v}_{h,1}^{n+1})=\Phi_{h}(\mathbf{v}_{h,2}^{n+1})$. Using \reff{202002034}, we have
\begin{eqnarray}\label{202003045}
  &&0\leq a_{h}(\mathbf{v}_{h,1}^{n+1},\mathbf{v}_{h,1}^{n+1}-\mathbf{v}_{h,2}^{n+1})
  -a_{h}(\mathbf{v}_{h,2}^{n+1},\mathbf{v}_{h,1}^{n+1}-\mathbf{v}_{h,2}^{n+1})\\
  &&~~=a_{h}^{lin}(\Phi_{h}(\mathbf{v}_{h,1}^{n+1})-\Phi_{h}(\mathbf{v}_{h,2}^{n+1}),\mathbf{v}_{h,1}^{n+1}-\mathbf{v}_{h,2}^{n+1})\nonumber\\
   &&~~=0,\nonumber
\end{eqnarray}
which implies that  $\mathbf{v}_{h,1}^{n+1}=\mathbf{v}_{h,2}^{n+1}$. Thus, we see that the above hypothesis does not hold, so\ $\Phi_{h}$ is a injective mapping.

(ii) Prove that $\Phi_{h}$ is surjective. From \cite{Deimling}, we know that $\Psi$ is surjective if\ $(E,(\cdot ,\cdot)_{E})$ is a \ Euclidean space and\ $\Psi: E\rightarrow E$ is a continuous mapping such that\ $\frac{(\Psi(x) ,x)_{E}}{\|x\|_{E}}\rightarrow +\infty$ as\ $\|x\|_{E}\rightarrow +\infty$. Since
\ $(V_{h},a_{h}^{lin}(\cdot ,\cdot))$ is a\ Euclidean space, from the definition of\ $a_{h}^{lin}$ and~\eqref{202002032} yields:
\begin{equation}\label{202003046}
  a_{h}^{lin}(\Phi_{h}(\mathbf{v}_{h}^{n+1}),\mathbf{v}_{h}^{n+1})\geq a_{h}^{lin}(\mathbf{v}_{h}^{n+1},\mathbf{v}_{h}^{n+1}),~\forall \mathbf{v}_{h}^{n+1}\in V_{h}.
\end{equation}

From (i) and (ii), we claim that  $\Phi_{h}$ is an isomorphism. Therefore, there exists a unique  $\mathbf{u}_{h}^{n+1}\in V_{h}$ such that $\Phi_{h}(\mathbf{u}_{h}^{n+1})=y_{h}^{n+1}$. Due to the uniqueness of  $y_{h}^{n+1}$, we see that  $(\mathbf{u}_{h}^{n+1}, {\xi_{h}^{n+1}},  {\eta_{h}^{n+\theta}})$ is the unique solution of the discrete problem \reff{e3.1}-\reff{e3.4}. The proof is finished.
\end{proof}

\subsection{Convergence analysis}\label{sec-3.4}
To derive the error estimates, we need to list
some basic results, one can refer to \cite{bs08, brezzi,cia}.
We firstly recall the following inverse inequality:
\begin{alignat}{2}\label{e3.13-00}
\|\nabla\varphi_h\|_{L^2(K)}\leq c_3h^{-1} \|\varphi_h\|_{L^2(K)}
\qquad \forall\varphi_h\in P_r(K), K\in T_h.
\end{alignat}

For any $\varphi\in L^2(\Omega)$, we define its $L^2$-projection $\mathcal{Q}_h: L^2\rightarrow W_h$ as
\begin{align}\label{e3.13-01}
\bigl( \mathcal{Q}_h\varphi, \psi_h  \bigr)=\bigl( \varphi, \psi_h  \bigr) \qquad \psi_h\in W_h.
\end{align}

It is well known that the projection operator $\mathcal{Q}_h: L^2\rightarrow W_h$ satisfies
that for any $\varphi\in H^s(\Omega) (s\geq1)$ \cite{bs08},
\begin{align}\label{e3.13-02}
\|\mathcal{Q}_h\varphi-\varphi\|_{L^2(\Ome)}+h\| \nabla(\mathcal{Q}_h\varphi
-\varphi) \|_{L^2(\Ome)}\leq Ch^\ell\|\varphi\|_{H^\ell(\Ome)}, \quad \ell=\min\{2, s\}.
\end{align}

We would like to point out that if $W_h\not\subset H^1(\Omega)$, the second term on the left-hand side
of \reff{e3.13-02} has to be replaced by the broken $H^1$-norm.

Next, for any $\varphi\in H^1(\Omega)$, we define its elliptic projection $\mathcal{S}_h\varphi$ by
\begin{alignat}{2}\label{e3.13-03}
\bigl(K\nabla\mathcal{S}_h\varphi, \nabla\varphi_h\bigr) &=\bigl(K\nabla\varphi, \nabla\varphi_h\bigr)
&& \qquad \forall \varphi_h\in W_h,\\
\bigl(\mathcal{S}_h\varphi, 1\bigr) &=\bigl(\varphi, 1\bigr).&&\label{e3.13-04}
\end{alignat}

It is well known that the projection operator $\mathcal{S}_h: H^1(\Omega)\rightarrow W_h$
satisfies that for any $\varphi\in H^s(\Omega) (s>1)$ \cite{bs08}, there holds
\begin{align}\label{e3.13-05}
\|\mathcal{S}_h\varphi-\varphi\|_{L^2(\Ome)}+h\| \nabla(\mathcal{S}_h\varphi-\varphi) \|_{L^2(\Ome)}
\leq Ch^\ell\|\varphi\|_{H^\ell(\Ome)}, \quad \ell=\min\{2, s\}.
\end{align}

Finally, for any $\bv\in \bH^1_\perp(\Omega)$, we define its elliptic projection $\mathcal{R}_h\bv$ by
\begin{alignat}{2}\label{e3.13-06}
\bigl(\varepsilon(\mathcal{R}_h\bv), \varepsilon(\bw_h)\bigr)
=\bigl(\varepsilon(\bv), \varepsilon(\bw_h)\bigr) \quad \bw_h\in \bV_h,
\end{alignat}

It is easy to show that the projection $\mathcal{R}_h\bv$ satisfies (cf. \cite{bs08}) that for any
$\bv\in \bH^1_\perp(\Omega)\cap \bH^s(\Omega) (s>1)$, there holds
\begin{eqnarray}\label{e3.13-07}
\qquad \|\mathcal{R}_h\bv-\bv\|_{L^2(\Ome)}+h\| \nabla(\mathcal{R}_h\bv-\bv) \|_{L^2(\Ome)}
\leq Ch^m\|\bv\|_{H^m(\Ome)}, m=\min\{3, s\}.
\end{eqnarray}

\begin{lemma} There exists a positive constant $C_{lp}$ such that
\begin{equation}\label{202003221}
\|  \mathcal{N}(\varepsilon(\mathbf{v}))
-\mathcal{N}(\varepsilon(\mathcal{R}_h(\mathbf{v})))\|_{L^{2}(\Omega)}\leq C_{lp}h^{m-1}\|\mathbf{v}\|_{H^{m}(\Omega)}, \quad m=\min\{3, s\}.
\end{equation}
\end{lemma}
\begin{proof}
Using  \reff{202002033}, \reff{e3.13-07} and the triangle inequality, we have
\begin{align}
\|  \mathcal{N}(\varepsilon(\mathbf{v}))
-\mathcal{N}(\varepsilon(\mathcal{R}_h(\mathbf{v})))\|_{L^{2}(\Omega)}
&\leq C_{lp}\| \varepsilon(\mathbf{v})
-\varepsilon(\mathcal{R}_h(\mathbf{v}))\|_{L^{2}(\Omega)} \no \\
&=C_{lp}\| \varepsilon(\mathbf{v}-\mathcal{R}_h(\mathbf{v}))\|_{L^{2}(\Omega)} \no \\
&\leq CC_{lp}\|  \nabla(\mathbf{v}-\mathcal{R}_h(\mathbf{v}))\|_{L^{2}(\Omega)} \no \\
&\leq  C_{lp}h^{m-1}\|\mathbf{v}\|_{H^{m}(\Omega)},\no
\end{align}
which implies that \reff{202003221} holds. The proof is complete.
\end{proof}

%

Also, we introduce the following symbols
\begin{align*}
	E_{\bu}^n &=\bu(t_n)-\mathcal{R}_h(\bu(t_n))+\mathcal{R}_h(\bu(t_n))-\bu_h^n
	:=\Lambda_{\bu}^n+\Theta_{\bu}^n,\\ 
	E_\xi^n &=\xi(t_n)-\mathcal{S}_h(\xi(t_n)) +\mathcal{S}_h(\xi(t_n))-\xi_h^n
	:=\Lambda_{\xi }^n+\Theta_{\xi }^n,\\
	E_\xi^n &=\xi(t_n)-\mathcal{Q}_h(\xi(t_n)) +\mathcal{Q}_h(\xi(t_n))-\xi_h^n
	:=\Pi_{\xi }^n+\Phi_{\xi }^n,\\
	E_\eta^n &=\eta(t_n)-\mathcal{S}_h(\eta(t_n))+\mathcal{S}_h(\eta(t_n))-\eta_h^n
	:=\Lambda_{\eta }^n+\Theta_{\eta }^n,\\
	E_\eta^n &=\eta(t_n)-\mathcal{Q}_h(\eta(t_n))+\mathcal{Q}_h(\eta(t_n))-\eta_h^n
	:=\Pi_{\eta }^n+\Phi_{\eta }^n,\\
	E_p^n &=p(t_n)-\mathcal{S}_h(p(t_n))+\mathcal{S}_h(p(t_n))-p_h^n
	:=\Lambda_{p }^n+\Theta_{p }^n.\\
	E_p^n &=p(t_n)-\mathcal{Q}_h(p(t_n))+\mathcal{Q}_h(p(t_n))-p_h^n
	:=\Pi_{p }^n+\Phi_{p }^n.
\end{align*}
\begin{lemma}\label{lma3.3}
Assume that $\{ (\bu_h^n, \xi_h^n, \eta_h^n) \}_{n\geq0}$ is generated by Algorithm \ref{alg210206-1} and
$\Lambda_{\bu}^n, \Theta_{\bu}^n, \Lambda_{\xi }^n, \Theta_{\xi }^n, \Lambda_{\eta }^n$
and $\Theta_{\eta }^n$ are defined as above. Then there holds the following estimate:
\begin{align}\label{e3.15-00}
&\mathcal{E}_h^\ell +\Delta t\sum_{n=1}^\ell \Bigl[\frac{1}{\mu_f}
\bigl( K\nab\hat{\Theta}_p^{n+1}-K\rho_f\bg,\nab\hat{\Theta}_p^{n+1}\bigr)\\
&\qquad
+\kappa_2\| d_t\Phi_\eta^{n+\theta} \|_{L^2(\Ome)}^2
+\kappa_3\| d_t\Phi_\xi^{n+1} \|_{L^2(\Ome)}^2 \Bigr) \Bigr] \no\\
&\leq\mathcal{E}_h^0+\Delta t\sum_{n=1}^\ell \Bigl[\bigl( \Pi_\xi^{n+1}, \div  \Theta_{\bu}^{n+1}\bigr)
-\bigl( \div d_t \Lambda_{\bu}^{n+1}, \Phi_\xi^{n+1} \bigr) \Bigr]\no\\
&\qquad
+\Delta t\sum_{n=1}^\ell
\Bigl[\bigl(\Phi_\xi^{n+1}, \div  \Theta_{\bu}^{n+1}\bigr)
-\bigl( \div d_t \Theta_u^{n+1}, \Phi_\xi^{n+1} \bigr) \Bigr] \no\\
&\qquad
+(1-\theta)(\Delta t)^2\sum_{n=1}^\ell\bigl( d_t^2 \eta_h(t_{n+1}), \Phi_\xi^{n+1} \bigr)
+\Delta t\sum_{n=1}^\ell \bigl( R_h^{n+1}, \hat{\Theta}_p^{n+1} \bigr)  \no\\
&\qquad
+\Delta t\sum_{n=1}^\ell\bigl(-\mathcal{N}(\varepsilon(\mathbf{u}(t_{n+1})))
+\mathcal{N}(\varepsilon(\mathcal{R}_h(\mathbf{u}(t_{n+1})))),\varepsilon(d_{t}\Theta^{n+1}_\mathbf{u})\bigr)\no\\
&\qquad
+(1-\theta)(\Delta t)^2 \sum_{n=1}^\ell \frac{\kappa_1}{\mu_f} \bigl( Kd_t\nabla\Theta_{\xi}^{n+1},
\nabla\hat{\Theta}_p^{n+1} \bigr),\no
\end{align}
where
\begin{alignat}{2}\label{e3.15-01}
\hat{\Theta}_p^{n+1} &:=\kappa_1\Theta_{\xi}^{n+1}+\kappa_2\Theta_{\eta}^{n+\theta},\\
\mathcal{E}_h^\ell &:=\frac{1}{2} \Bigl[ 2\Delta tC_{mn}\|\varepsilon(\Theta_{\mathbf{u}}^{l+1})\|_{L^{2}(\Omega)}^{2}+\kappa_2\|\Phi_\eta^{\ell+\theta} \|_{L^2(\Ome)}^2+\kappa_3\|\Phi_\xi^{\ell+1}\|_{L^2(\Ome)}^2 \Bigr],\label{e3.15-01-0}\\
R_h^{n+1} &:=-\frac{1}{\Delta t}\int_{t_{n}}^{t_{n+1}}(s-t_{n})\eta_{tt}(s)\,ds.\label{e3.15-02}
\end{alignat}
\end{lemma}
\begin{proof}
Subtracting \reff{e3.1} from \reff{e2.4}, \reff{e3.2} from \reff{e2.5},  \reff{e3.4} from \reff{e2.6},
respectively, we get the following equations:
\begin{alignat}{2}\label{e3.15}
&\bigl(\mathcal{N}(\varepsilon(\bu(t_{n+1})))-\mathcal{N}(\varepsilon(\bu^{n+1}_{h})), \vepsi(\bv_h) \bigr)-\bigl( E_\xi^{n+1}, \div \bv_h \bigr)= 0
&&\quad \forall \bv_h\in V_h, \\
&\kappa_3\bigl(E_{\xi}^{n+1},\varphi_h\bigr)+\bigl(\div E^{n+1}_{\bu}, \varphi_h \bigr) && \label{e3.16} \\
&\hskip 0.9in
= \kappa_1\bigl( E^{n+\theta}_{\eta}, \varphi_h \bigr)+(1-\theta)\Delta t\bigl(  d_t\eta(t_{n+1}), \varphi_h  \bigr)
&&\quad \forall \varphi_h \in M_h, \no \\
&\bigl(d_tE_\eta^{n+1}, \psi_h \bigr)
+\frac{1}{\mu_f} \bigl(K\nab (\kappa_1E_\xi^{n+1} +\kappa_2E_\eta^{n+1})-K\rho_f\bg, \nab \psi_h \bigr)
&&  \label{e3.18}\\
&\hskip 1.9in
=(R^{n+1}_h, \psi_h) &&\quad\forall \psi_h \in W_h,\no\\
&  E_{\bu}^0=0,\quad E_\xi^0=0, \quad E_\eta^{-1}=0. &&\label{e3.19}
\end{alignat}

Using the definition of the projection operators $\mathcal{Q}_h, \mathcal{S}_h, \mathcal{R}_h$, we have
\begin{alignat}{2}\label{e3.20}
&\bigl(\mathcal{N}(\varepsilon(\bu(t_{n+1})))-\mathcal{N}(\varepsilon(\bu^{n+1}_{h})), \vepsi(\bv_h) \bigr)-\bigl( \Phi_\xi^{n+1}, \div \bv_h \bigr)\\
&\quad\quad\quad= (\Pi_{\xi}^{n+1}, \div \bv_h),
&&\quad\forall \bv_h\in V_h, \no\\
&\kappa_3\bigl(\Phi_{\xi}^{n+1}, \varphi_h \bigr) +\bigl(\div\Theta^{n+1}_{\bu}, \varphi_h \bigr)
=\kappa_1 \bigl(\Theta_\eta^{n+\theta}, \varphi_h \bigr) \label{e3.21}\\
&\hskip 0.8in
-\bigl( \div \Lambda_{\bu}^{n+1}, \varphi_h \bigr)
+(1-\theta)\Delta t\bigl(  d_t\eta(t_{n+1}), \varphi_h  \bigr)
&&\quad\forall \varphi_h \in M_h,\no
\end{alignat}
\begin{eqnarray}
&&\bigl(d_t\Phi_\eta^{n+1}, \psi_h \bigr) +\frac{1}{\mu_f} \bigl(K\nab (\kappa_1\Theta_\xi^{n+1}
+\kappa_2\Theta_\eta^{n+1})-K\rho_f\bg, \nab \psi_h \bigr)  \label{e3.23}\\
&&\hskip 1.5in
=\bigl( R^{n+1}_h, \psi_h \bigr) \quad\forall \psi_h \in W_h, \no\\
&&E_{\bu}^0=0,\quad E_\xi^0=0, \quad E_\eta^{-1}=0. \label{e3.24}
\end{eqnarray}

Thus, \eqref{e3.15-00} follows from setting $\bv_h=\Theta^{n+1}_{\bu}$ in \reff{e3.20}, $\varphi_h=\Theta^{n+1}_\xi$
(after applying the difference operator $d_t$ to the equation \reff{e3.21}),
$\psi_h=\hat{\Theta}_p^{n+1}:=\kappa_1\Theta_{\xi}^{n+1}+\kappa_2\Theta_{\eta}^{n+\theta}=\Pi^{n+1}_{p }-\Lambda^{n+1}_{p}+\kappa_1\Phi_\xi^{n+1} +\kappa_2\Phi_\eta^{n+\theta}$ in \reff{e3.23},
adding the resulting equations, and applying the summation operator $\Delta t\sum^{\ell}_{n=1}$ to both sides. The proof is complete.
\end{proof}

\begin{theorem}\label{thm1301}
Let $\{(u_h^{n}, \xi_h^{n}, \eta_h^{n})\}_{n\geq 0}$ be the solution of Algorithm \ref{alg210206-1}, then
there holds the error estimate for $\ell \leq N$
\begin{align}\label{e131130-1}
&\max_{0\leq n\leq \ell}\bigg[ \sqrt{C_{mn}}\|\varepsilon(\Theta_{u}^{n+1})\|_{L^2(\Ome)}
+\sqrt{\kappa_2}\|\Phi_\eta^{n+\theta}\|_{L^2(\Ome)}+\sqrt{\kappa_3}\|\Phi_\xi^{n+1}\|_{L^2(\Ome)}\bigg] \\
&\hskip 1.2in
+\bigg[\Delta t\sum_{n=0}^\ell\frac{K}{\mu_f}
\|\hat{\Theta}_p^{n+1}\|_{L^2}^2\bigg]^\frac{1}{2}
\leq C_1(T)\Delta t+C_2(T)h^2\no
\end{align}
provided that $\Delta t =O(h^2)$ when $\theta=0$ and $\Delta t >0$ when $\theta=1$, where
\begin{align}
C_1(T)&= C\|\eta_t\|_{L^2((0, T); L^2(\Ome))}^2+ C\|(\eta)_{tt}\|_{L^2((0, T); H^{1}(\Ome)')},\\
C_2(T)&=C\|\xi\|_{L^{\infty}((0,T);H^2(\Ome))} +C\|\xi_t\|_{L^2((0,T);H^2(\Ome))} \\
&
+C\|\mathcal{N}(\varepsilon(\mathbf{u}))\|_{L^2((0,T);H^2(\Omega))}+C\|\bu\|_{L^2((0,T);H^3(\Ome))}\no\\
&+C\|\div({\bu})_t\|_{L^2((0,T);H^2(\Ome))}. \no
\end{align}
\end{theorem}
\begin{proof}
Using the fact of $\Theta_{\bu}^0=\mathbf{0}$, $\Theta_\xi^0=0$, $\Theta_\eta^{-1}=0$ and \reff{e3.15-00}, we get
\begin{align}\label{e131204-00}
&\frac{1}{2} \Bigl[2\Delta tC_{mn}\|\vepsi(\Theta^{\ell+1}_\bu)\|_{L^2(\Ome)}^2+\kappa_2\|\Phi_{\eta}^{\ell+\theta} \|_{L^2(\Ome)}^2+\kappa_3\|\Phi_{\xi}^{\ell+1}\|_{L^2(\Ome)}^2 \Bigr]\\
&+\Delta t\sum_{n=1}^\ell\bigl[\frac{1}{\mu_f} \bigl(K\nabla\hat{\Theta}_p^{n+1}-K\rho_f\bg,
\nabla\hat{\Theta}_p^{n+1} \bigr)\no\\
&
+\frac{\Delta t}{2}\Bigl(
\kappa_2\| d_t\Phi_{\eta}^{n+\theta} \|_{L^2(\Ome)}^2+\kappa_3\| d_t\Phi_{\xi}^{n+1} \|_{L^2(\Ome)}^2 \Bigr) \bigr]\no\\
&
\leq (1-\theta)(\Delta t)^2\sum_{n=1}^\ell\bigl( d_t^2 \eta(t_{n+1}), \Phi_{\xi}^{n+1} \bigr)\no\\
&
+\Delta t\sum_{n=1}^\ell \bigl( R_h^{n+1}, \hat{\Theta}_p^{n+1} \bigr)+C_{mn}\|\vepsi(\Theta^{1}_\bu)\|_{L^2(\Ome)}^2\no\\
&
+\Delta t\sum_{n=1}^\ell\bigl(-\mathcal{N}(\varepsilon(\bu(t_{n+1})))+\mathcal{N}(\varepsilon(\mathcal{R}_h(\bu(t_{n+1})))),\vepsi(d_{t}\Theta^{n+1}_\bu)\bigr)\no\\
&
+\Delta t\sum_{n=1}^\ell \Bigl[\bigl( \Pi_\xi^{n+1}, \div \Theta_{\bu}^{n+1}\bigr)
-\bigl( \div d_t \Lambda_{\bu}^{n+1}, \Phi_\xi^{n+1} \bigr)\Bigr] \no\\
&
+\Delta t\sum_{n=1}^\ell
\Bigl[\bigl(\Phi_\xi^{n+1}, \div  \Theta_{\bu}^{n+1}\bigr)
-\bigl( \div d_t \Theta_u^{n+1}, \Phi_\xi^{n+1} \bigr) \Bigr]\no\\
&
+(1-\theta)(\Delta t)^2\sum_{n=1}^\ell\frac{\kappa_1}{\mu_f} \bigl(Kd_t\nab\Theta_{\xi}^{n+1}, \nab\hat{\Theta}_p^{n+1}\bigr).\no
\end{align}
We now estimate each term on the right-hand side of \reff{e131204-00}. To bound the first term ($\theta=0$) on the right-hand side of \reff{e131204-00}, we use the summation by parts formula and $d_t\eta_h(t_0)=0$ and get
\begin{align}\label{e131209-00}
\sum_{n=0}^\ell \bigl( d_t^2 \eta(t_{n+1}), \Phi_{\xi}^{n+1} \bigr)
=\frac{1}{\Delta t}\bigl( d_t \eta(t_{l+1}),  \Phi_{\xi}^{l+1}\bigr)
-\sum_{n=1}^\ell \bigl( d_t \eta(t_{n}), d_t\Phi_{\xi}^{n+1} \bigr).
\end{align}
Now, we can get the estimation of the first term on the right-hand side of \reff{e131209-00} as follows:
\begin{align}\label{e131209-01}
&\frac{1}{\Delta t}\bigl( d_t \eta(t_{\ell+1}), \Phi_{\xi}^{\ell+1}\bigr)
\leq\frac{1}{\Delta t}\|d_t \eta(t_{\ell+1})\|_{L^2(\Ome)}\| \Phi_{\xi}^{\ell+1}\|_{L^2(\Ome)}\\
&\leq \frac{1}{\Delta t}\|\eta_t\|_{L^2((t_\ell, t_{\ell+1}); L^2(\Ome))}\cdot \frac{1}{\beta_1}
\sup_{\bv_h\in\bV_h}[\frac{\bigl(\mathcal{N}(\varepsilon(\mathcal{R}_h(\bu(t_{l+1})))), \varepsilon(\bv_h)  \bigr)}{\|\nabla \bv_h\|_{L^2(\Ome)}}\no\\
&
-\frac{(\mathcal{N}(\varepsilon(\bu^{l+1}_{h})), \varepsilon(\bv_h)  \bigr)}{\|\nabla \bv_h\|_{L^2(\Ome)}}-\frac{(\Pi_{\xi}^{\ell+1},\div \bv_h)}{\|\nabla \bv_h\|_{L^2(\Ome)}}]\no\\
&
\leq\frac{C_{mn}}{4(\Delta t)^2} \|\varepsilon(\Theta_{\bu}^{\ell+1})\|_{L^2(\Ome)}^{2}
+\frac{C_{mn}}{4(\Delta t)^2} \|\varepsilon(\Lambda_{\bu}^{\ell+1})\|_{L^2(\Ome)}^{2}
\no\\
&+\frac{C}{\beta_1^{2}} \|\eta_t\|_{L^2((t_\ell, t_{\ell+1}); L^2(\Ome))}^2+\frac{C}{C_{lp}\beta_1^{2}}\|\Pi_{\xi}^{\ell+1}\|_{L^2(\Ome)}^2.\no
\end{align}
As for the second term on the right-hand side of \reff{e131209-00}, we have
\begin{align}\label{202005301}
&\sum_{n=1}^\ell\bigl( d_t \eta(t_{n}), d_t\Phi_{\xi}^{n+1} \bigl)
\leq \sum_{n=1}^\ell\| d_t \eta(t_{n})\|_{L^2(\Ome)}\| d_t\Phi_{\xi}^{n+1}\|_{L^2(\Ome)}\\
&
\leq \sum_{n=1}^\ell|| d_t \eta(t_{n})\|_{L^2(\Ome)}\cdot\frac{1}{\beta_1}
\sup_{\bv_h\in\bV_h}[\frac{\bigl( d_{t}\mathcal{N}(\varepsilon(\mathcal{R}_h(\bu(t_{l+1})))), \varepsilon(\bv_h)  \bigr)}{\|\nabla \bv_h\|_{L^2(\Ome)}}\no\\
&
-\frac{\bigl( d_{t}\mathcal{N}(\varepsilon(\bu^{l+1}_{h})), \varepsilon(\bv_h)  \bigr)}{\|\nabla \bv_h\|_{L^2(\Ome)}}
-\frac{(d_t\Pi_{\xi}^{n+1},\div \bv_h)}{\|\nabla \bv_h\|_{L^2(\Ome)}}]\no\\
&
\leq \frac{C}{\beta_1} \sum_{n=1}^\ell \| d_t \eta(t_{n})\|_{L^2(\Ome)}
\Bigl[ C_{lp}\| d_t\varepsilon(\Theta_{\bu}^{n+1})\|_{L^2(\Ome)}\no\\
&
+ C_{lp}\| d_t\varepsilon(\Lambda_{\bu}^{n+1})\|_{L^2(\Ome)}
+\|d_t\Pi_{\xi}^{n+1}\|_{L^2(\Ome)}  \Bigr]\no\\
&+\frac{C}{C_{lp}\beta_1^{2}} \|d_{t}\Pi_{\xi}^{n+1}\|_{L^2(\Ome)}^2\Bigr]
+\frac{C}{\beta_1^{2}}\|\eta_t\|_{L^2((0, T); L^2(\Ome))}^2.\no
\end{align}

The second term on the right-hand side of \reff{e131204-00} can be bounded as
\begin{align}
\bigl|\bigl( R_h^{n+1}, \hat{\Theta}_p^{n+1} \bigr) \bigr|
&\leq \|R_h^{n+1}\|_{H^{1}(\Ome)'}\|\nabla\hat{\Theta}_p^{n+1}\|_{L^2(\Ome)} \label{e131209-02}\\
&\leq \frac{K}{4\mu_f}\|\nabla\hat{\Theta}_p^{n+1}\|_{L^2(\Ome)}^2
        +\frac{\mu_f}{K}\|R_h^{n+1}\|_{H^{1}(\Ome)'}^2\no\\
&\leq \frac{K}{4\mu_f}\|\nabla\hat{\Theta}_p^{n+1}\|_{L^2(\Ome)}^2
        +\frac{\mu_f\Delta t}{3K}\|\eta_{tt}\|_{L^2((t_{n}, t_{n+1}); H^{1}(\Ome)')}^2,\no
\end{align}
where we have used the fact that
\begin{align*} 
\|R_h^{n+1}\|_{H^{1}(\Ome)'}^2 \leq\frac{\Delta t}{3}\int_{t_{n}}^{t_{n+1}} \|\eta_{tt}\|_{H^{1}(\Ome)'}^2\,dt.
\end{align*}

The third term on the right-hand side  of \reff{e131204-00} can be bounded by
\begin{eqnarray}\label{2020020310}
&& \Delta t\sum_{n=1}^\ell\bigl(-\mathcal{N}(\varepsilon(\mathbf{u}(t_{n+1})))
+\mathcal{N}(\varepsilon(\mathcal{R}_h(\mathbf{u}(t_{n+1})))),\varepsilon(d_{t}\Theta^{n+1}_\mathbf{u})\bigr)\\
  && \leq \Delta t\sum_{n=1}^\ell \bigl[      \frac{2C_{lp}}{C_{mn}}\|\mathcal{N}(\varepsilon(\mathbf{u}(t_{n+1})))
-\mathcal{N}(\varepsilon(\mathcal{R}_h(\mathbf{u}(t_{n+1}))))\|_{L^2(\Omega)}^{2}\nonumber\\
   &&+\frac{C_{mn}C_{lp}}{8}\|d_{t}\varepsilon(\Theta_{\mathbf{u}}^{n+1})\|_{L^2(\Omega)}^{2}\bigr],\nonumber
\end{eqnarray}

The fourth term on the right-hand side  of \reff{e131204-00} can be bounded by
\begin{align}\label{e141104-01}
&\Delta t\sum_{n=1}^\ell \Bigl[ \bigl( \Pi_{\xi}^{n+1}, \div \Theta_{\bu}^{n+1}\bigr)
-\bigl( \div d_t \Lambda_{\bu}^{n+1}, \Phi_{\xi}^{n+1} \bigr)\Bigl]\\
&\quad
\leq \frac12\Delta t\sum_{n=1}^\ell\bigr(\|\Pi_{\xi}^{n+1}\|_{L^2(\Ome)}^2
+ C\|\varepsilon(\Theta_{\bu}^{n+1})\|_{L^2(\Ome)}^2\bigr)\no\\
&\quad
+\frac12\Delta t\sum_{n=1}^\ell\bigr(\|\div d_t \Lambda_{\bu}^{n+1}\|_{L^2(\Ome)}^2+ \|\Phi_{\xi}^{n+1}\|_{L^2(\Ome)}^2 \bigr).\notag
\end{align}
The fifth term on the right-hand side of \reff{e131204-00} can be bounded by
\begin{align}
&\Delta t\sum_{n=1}^\ell \Bigl[ \bigl( \Phi_{\xi}^{n+1}, \div \Theta_{\bu}^{n+1}\bigr)
-\bigl( \div d_t \Lambda_{\bu}^{n+1}, \Phi_{\xi}^{n+1} \bigr)\Bigl]\\
&\quad
\leq \frac12\Delta t\sum_{n=1}^\ell\bigr(\|\Phi_{\xi}^{n+1}\|_{L^2(\Ome)}^2
+ C\|\varepsilon(\Theta_{\bu}^{n+1})\|_{L^2(\Ome)}^2\bigr)\no\\
&\quad
+\frac12\Delta t\sum_{n=1}^\ell\bigr(\|\div d_t \Theta_{\bu}^{n+1}\|_{L^2(\Ome)}^2+ \|\Phi_{\xi}^{n+1}\|_{L^2(\Ome)}^2 \bigr)\no
\end{align}
by using the fact of $
\|\div \Theta_{\bu}^{n}\|_{L^2(\Ome)} \leq C\|\varepsilon(\Theta_{\bu}^{n})\|_{L^2(\Ome)}$.

If $\theta=0$, we also need to bound the last term on the right-hand side of \reff{e131204-00},
which is carried out below:
\begin{align}\label{e141104-02}
&\sum_{n=1}^\ell \bigl(d_t\nabla\Theta_{\xi}^{n+1}, \nabla\hat{\Theta}_p^{n+1} \bigr)
\leq\sum_{n=1}^\ell\|d_t\Theta_{\xi}^{n+1}\|_{L^2(\Ome)}\|\nabla\hat{\Theta}_p^{n+1}\|_{L^2(\Ome)}\\
&\,
\leq\sum_{n=1}^\ell \sup_{\bv_h\in \bV_h}
\frac{\bigl(d_{t}\mathcal{N}(\varepsilon(\mathcal{R}_h(\bu(t_{l+1}))))-d_{t}\mathcal{N}(\varepsilon(\bu^{l+1}_{h})), \varepsilon(\bv_h)  \bigr)}{\|\nabla \bv_h\|_{L^2(\Ome)}}\times \no \\
&\|\nabla\hat{\Theta}_p^{n+1}\|_{L^2(\Ome)}-\frac{\bigl(d_t\Lambda_{\xi}^{n+1},\div \bv_h\bigl)}{\|\nabla \bv_h\|_{L^2(\Ome)}}
\,\|\nabla\hat{\Theta}_p^{n+1}\|_{L^2(\Ome)}\no\\
&\,
\leq\sum_{n=1}^\ell \Bigl[\frac{C_{lp}\kappa_1\Delta t}{h^2\beta_1^2}m\mu^{2}\|d_{t}\varepsilon(\Theta_{\bu}^{\ell+1})\|_{L^2(\Ome)}^{2}
+\frac{C_{lp}\kappa_1\Delta t}{h^2\beta_1^2}m\mu^{2}\|d_{t}\varepsilon(\Lambda_{\bu}^{\ell+1})\|_{L^2(\Ome)}^{2}\no\\
&+\frac{\kappa_1\Delta t}{C_{lp}h^2\beta_1^2} \|d_t\Lambda_{\xi}^{n+1}\|_{L^2}^{2}
+\frac{\kappa_1^{-1}}{4\Delta t}\|\nabla\hat{\Theta}_p^{n+1}\|_{L^2(\Ome)}^2 \Bigr].\no
\end{align}
Substituting \reff{e131209-00}-\reff{e141104-02} into \reff{e131204-00} and rearranging all of terms, we get
\begin{eqnarray}\label{e131209-04}
&&\mu \|\varepsilon(\Theta_{\bu}^{\ell+1})\|_{L^2(\Ome)}^2
+\kappa_2\|\Phi_{\eta}^{\ell+\theta}\|_{L^2(\Ome)}^2+\kappa_3\|\Phi_{\xi}^{\ell+1}\|_{L^2(\Ome)}^2 
+\Delta t \sum_{n=1}^\ell\frac{K}{\mu_f} \|\nabla\hat{\Theta}_p^{n+1} \|_{L^2(\Ome)}^2 \no\\
&&
\leq \frac{4\mu_f(\Delta t)^2}{3K} \|\eta_{tt}\|_{L^2((0, T); H^{1}(\Omega)')}^2
+\frac{4m\mu(\Delta t)^2}{\beta_1^2} \|\eta_t\|_{L^2((0,T);L^2(\Ome))}^2, \no\\
&&
+C_1\|\Pi_{\xi}^{\ell+1}\|_{L^2(\Ome)}^2
+\Delta t\sum_{n=1}^\ell \|d_t\Pi_{\xi}^{n+1}\|_{L^2(\Ome)}^2+\Delta t\sum_{n=1}^\ell \|\div d_t \Lambda_{\bu}^{n+1}\|_{L^2(\Ome)}^2, \no\\
&&
+\Delta t\sum_{n=1}^l \frac{2C_{lp}}{C_{mn}}\|\mathcal{N}(\varepsilon(\mathbf{u}(t_{n+1})))
-\mathcal{N}(\varepsilon(\mathcal{R}_h(\mathbf{u}(t_{n+1}))))\|_{L^2(\Omega)}^{2}\no
\end{eqnarray}
provide that
$\Delta t\leq C_{mn}\mu_f\beta_1^2(8C_{lp}\kappa_1^2 K)^{-1} h^2$ in the case of $\theta=0$, but
it holds for all $\Delta t>0$ in the case of $\theta=1$. Hence, \reff{e131130-1} follows by using the
approximation properties of the projection operators $\mathcal{Q}_h, \mathcal{R}_h$ and
$\mathcal{S}_h$. The proof is complete.
\end{proof}

Using Theorem \ref{thm1301}, we have the following main result.
\begin{theorem}\label{thm3.5}
The solution of Algorithm \ref{alg210206-1} satisfies the following error estimates:
\begin{align}\label{e131209-05}
&\max_{0\leq n\leq N} \Bigl[ \sqrt{C_{mn}}\|\nabla(u(t_{n}))-\nabla(u_{h}^{n})\|_{L^2(\Ome)}
+\sqrt{\kappa_2} \|\eta(t_n)-\eta_h^n\|_{L^2(\Ome)}\\
&\hskip .5in
+\sqrt{\kappa_3} \|\xi(t_n)-\xi_h^n\|_{L^2(\Ome)} \Bigr]
\leq \widehat{C}_1(T) \Delta t +\widehat{C}_2(T)h^2,\no \\
&\bigg[ \Delta t \sum_{n=0}^N \frac{K}{\mu_f} \|\nabla(p(t_n)-p_h^n) \|_{L^2(\Ome)}^2 \bigg]^{\frac12}
\leq \widehat{C}_1(T) \Delta t +\widehat{C}_2(T)h \label{E_p}
\end{align}
provided that $\Delta t=O(h^2)$ if $\theta=0$ and $\Delta t>0$ if $\theta=1$. Here
\begin{align*}
\widehat{C}_1(T)&:=C_1(T),\\
\widehat{C}_2(T)&:=C_2(T)+\|\xi\|_{L^{\infty}((0,T);H^2(\Ome))}+\|\eta\|_{L^{\infty}((0,T);H^2(\Ome))} \\
&
+\|\nabla\bu\|_{L^{\infty}((0,T);H^2(\Ome))}.
\end{align*}
\end{theorem}
\begin{proof}
The above estimations follow immediately from an application of the triangle inequality on
\begin{alignat*}{2}
\bu(t_n)-\bu_h^n &=\Lambda_{\bu}^n+\Theta_{\bu}^n,\\
\xi(t_n)-\xi_h^n &=\Lambda_{\xi }^n+\Theta_{\xi }^n, \\
\eta(t_n)-\eta_h^n &=\Lambda_{\eta }^n+\Theta_{\eta }^n,\\
p(t_n)-p_h^n &=\hat{\Lambda}_p^n + \hat{\Theta}_p^n.
\end{alignat*}
and appealing to \reff{e3.13-02}, \reff{e3.13-07} and Theorem \ref{thm1301}.
\end{proof}
\section{Numerical tests}\label{sec-4}
~

\medskip
{\bf Test 1.} Let $\Omega=[0,1]\times [0,1]$, $\Gamma_1=\{(1,x_2); 0\leq x_2\leq 1\}$,
$\Gamma_2=\{(x_1,0); 0\leq x_1\leq 1\}$,  $\Gamma_3=\{(0,x_2); 0\leq x_2\leq 1\}$,
$\Gamma_4=\{(x_1,1); 0\leq x_1\leq 1\}$, and $T=1$. The problem is \eqref{e2.6}-\eqref{e2.9}
with following source functions:

\begin{align*}
\sigma( \bu)&=\mu\varepsilon(\mathbf{u})+\mu \nabla^{T}\mathbf{u}\nabla\mathbf{u}+\lambda\|\nabla\mathbf{u}\|_{F}^{2}I+\lambda \div\mathbf{u}I,\\
\mathcal{N}(\varepsilon( \bu))&=\mu\varepsilon(\mathbf{u})+\mu \nabla^{T}\mathbf{u}\nabla\mathbf{u}+\lambda\|\nabla\mathbf{u}\|_{F}^{2}I,\\
\mathbf{f} &=-(\lambda+\mu) t(1,1)^T-2(\mu+\lambda) t^{2}(x_{1},x_{2})^{T}+\alpha \cos(x_1+x_2)e^t(1,1)^T,\\
\phi &=\Bigl(c_0+\frac{2\kappa}{\mu_f} \Bigr)\sin(x_1+x_2)e^t+\alpha(x_1+x_2),
\end{align*}
and the following boundary and initial conditions:
\begin{alignat*}{2}
p &= \sin(x_1+x_2)e^t  &&\qquad\mbox{on }\partial\Omega_T,\\
u_1 &= \frac12 x_1^2t &&\qquad\mbox{on }\Gamma_j\times (0,T),\, j=1,3,\\
u_2 &= \frac12 x_2^2t &&\qquad\mbox{on }\Gamma_j\times (0,T),\, j=,2,4,\\
\sigma\bf{n}-\alpha \emph{p}\bf{n} &= \mathbf{f}_1, &&\qquad \mbox{on } \p\Ome_T,\\
\mathbf{u}(x,0) = \mathbf{0},  \quad p(x,0) &=\sin(x_1+x_2) &&\qquad\mbox{in } \Ome,
\end{alignat*}
where
\begin{align*}
    \mathbf{f}_1(x,t)= \lambda(x_1+x_2) (n_1,n_2)^T t +\mu t(x_{1}n_{1},x_{2}n_{2})^{T}+\mu t^{2}\bigl(x_{1}^{2}n_{1},x_{2}^{2}n_{2}\bigr)^{T},\\
+\lambda t^{2}(x_{1}^{2}+x_{2}^{2})(n_1,n_2)^T-\alpha\sin(x_1+x_2)(n_1,n_2)^T e^t,
\end{align*}
the exact solution of this problem is
\[
\mathbf{u}(x,t)=\frac{t}2 \bigl( x_1^2, x_2^2 \bigr)^T,\qquad p(x,t)=\sin(x_1+x_2)e^t.
\]

The parameters are chosen as follows: $\lambda=0.00042$, $ \mu=0.0048$, $c_{0}= 0.00001$, $\alpha=0.83$, $\nu=0.04$, $ K =0.00001$ and $E=0.01$. Note that the boundary conditions used above are not pure Neumann conditions, instead, they are mixed Dirichlet-Neumann conditions, so the approach and methods of this paper also work in this case.

\begin{table}[!htbp]
\begin{center}
\caption{Spatial errors and convergence rates of $\mathbf{u} $ and $p$} \label{tab1}
\begin{tabular}{l c c c c}
\hline
$h$ & $\|\mathbf{u}-\mathbf{u}_h\|_{L^2(H^1)}$ & CR& $\|p-p_h\|_{L^2(H^1)}$ & CR\\ \hline
$h=0.176$ & 2.24e-4 &       & 1.41e-5  &      \\ 
$h=0.088$ & 3.87e-5 &2.5331 &  3.49e-6&2.0144     \\ 
$h=0.044$ & 6.79e-6 &2.5109 &  8.74e-7&1.9975  \\ 
$h=0.022$ & 1.19e-6 &2.5125       &2.19e-7& 1.9967\\ \hline
\end{tabular}
\end{center}
\end{table}


Table \ref{tab1} displays the error of displacement and pressure in\ $L^2(0,T;H^1(\Omega))$-norm, which are consistent with the theoretical result.

\begin{figure}[!htbp]
	\centering
	\includegraphics[height=2.5in,width=3.5in]{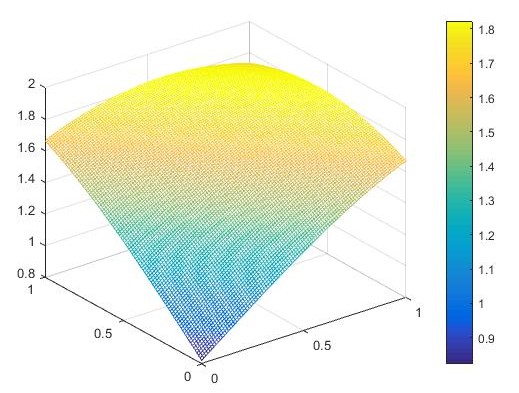}
	\caption{The numerical solution of the pressure $p_{h}^{n+1}$.}\label{figure_p4}
\end{figure}

\begin{figure}[!htbp]
	\centering
	\includegraphics[height=2.5in,width=3.5in]{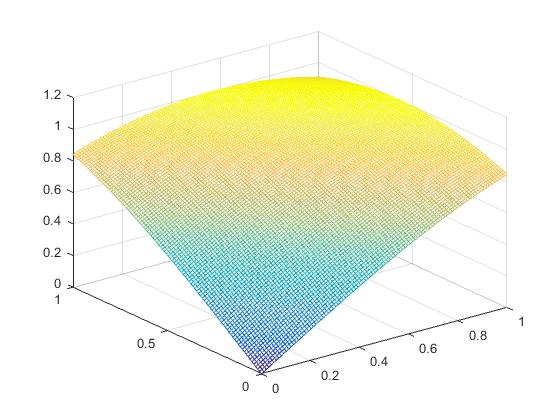}
	\caption{The analytical solution of the pressure $p$.}
	\label{figure_p4_add1}
\end{figure}

Figure \ref{figure_p4} shows the numerical solution of pressure $p_{h}^{n+1}$ at $T=1$, and Figure \ref{figure_p4_add1}  shows the analytical solution of pressure $p$. From the above figures, we can find that our method is stable and there is no ``locking phenomenon".


\medskip
{\bf Test 2.}  Let $\Omega= [0,1]\times[0,1]$, the definition of $\Gamma_{1}, \Gamma_{2}, \Gamma_{3}, \Gamma_{4}$ are same as {\bf Test 1}, take $\Delta t=1e-6$,  $T=1$.  Consider the problem \reff{e2.4}-\reff{e2.6} with the following source functions:
\begin{equation*}
  \sigma(\mathbf{u})=(4-2e^{-dev(\varepsilon(\mathbf{u}))})\varepsilon(\mathbf{u})+(1+e^{-dev(\varepsilon(\mathbf{u}))}+\lambda)\div\mathbf{u}I,
\end{equation*}
\begin{equation*}
  \mathcal{N}(\varepsilon(\mathbf{u}))=(4-2e^{-dev(\varepsilon(\mathbf{u}))})\varepsilon(\mathbf{u})+(1+e^{-dev(\varepsilon(\mathbf{u}))})\div\mathbf{u}I,
\end{equation*}
\begin{equation*}
  -dev(\varepsilon(\mathbf{u}))=tr(\varepsilon^{2}(\mathbf{u}))-\frac{1}{d}tr^{2}(\varepsilon(\mathbf{u})),
\end{equation*}
\begin{equation*}
  \phi=-\frac{c_{0}}{\pi}\sin(\pi x+\pi y)\\
  +2\alpha\pi t\sin(\pi x+\pi y)-\frac{2K\pi t}{\mu_{f}}\sin(\pi x+\pi y),
\end{equation*}
The analytical solution of the above problem is
\begin{eqnarray*}
        \mathbf{u} &=& t^{2}(\sin(\pi x)\sin(\pi y), \sin(\pi x)\sin(\pi y)), \\
        p &=& -\pi^{-1} t(\sin(\pi x)\cos(\pi y)+\cos(\pi x)\sin(\pi y)).
      \end{eqnarray*}
The parameters are chosen as follows: $\lambda=576.923$, $\mu=384.615$, $c_{0}= 0.5$, $\alpha=0.8$, $\nu=0.3$, $ K =0.5$ and $E=1000$.
\begin{table}[!htbp]
\begin{center}
	\caption{Spatial errors of $\mathbf{u}$ and $p$} \label{tab102}
\begin{tabular}{l c c c c }
\hline
$h$ & $\|\mathbf{u}-\mathbf{u}_h\|_{L^2(L^2)}$  &$ \|\mathbf{u}-\mathbf{u}_h\|_{L^2(H^1)}$ & $\|p-p_h\|_{L^2(L^2)}$ & $\|p-p_h\|_{L^2(H^1)}$\\ \hline
$h=0.176$ & 1.99711e-13  & 6.27422e-13   & 4.8716e-9       & 2.23502e-8      \\ 
$h=0.088$ & 1.99716e-13  & 6.27426e-13   & 5.01447e-9      &2.23502e-8       \\ 
$h=0.044$ & 1.99716e-13  & 6.27427e-13   & 5.05085e-9      &2.24583e-8        \\ 
$h=0.022$ & 1.99716e-13  & 6.27427e-13   & 5.05998e-9     &2.24922e-8        \\ \hline
\end{tabular}
\end{center}
\end{table}


Table \ref{tab102} displays the error of displacement $\mathbf{u}$ and the pressure $p$ with $L^2(0,T;$ $L^2(\Omega))$-norm and $L^2(0,T;H^1(\Omega))$-norm, which are consistent with the theoretical result.
\begin{figure}[!htbp]
	\centering
	\includegraphics[height=2.0in,width=2.8in]{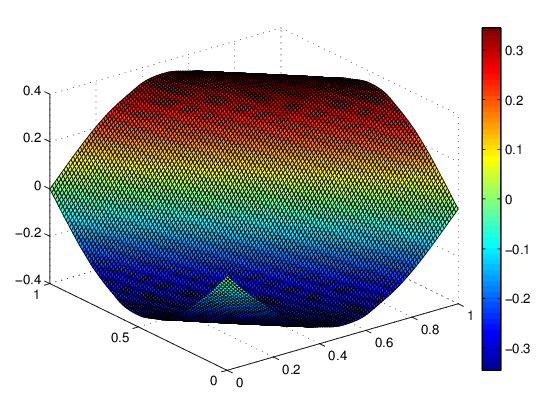}
	\caption{ The analytical solution of pressure $p$ at $T$.}\label{figure_p8}
\end{figure}
\begin{figure}[!htbp]
\centering
\includegraphics[height=2.2in,width=3in]{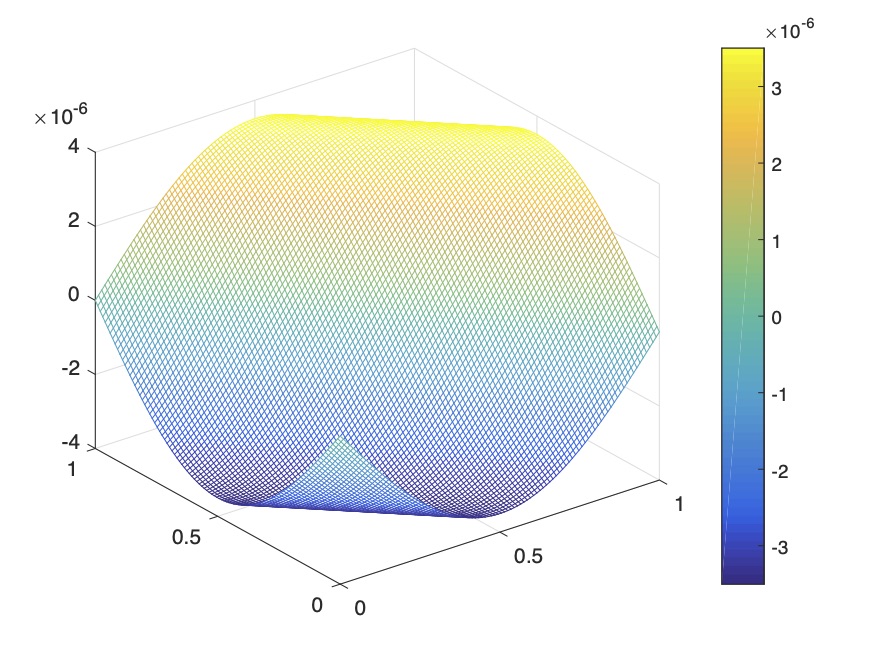}
\caption{The numerical pressure $p_{h}^{n+1}$ at the terminal time $T$.}\label{figure_p6}
\end{figure}

\begin{figure}[!htbp]
\centering
\includegraphics[height=2.0in,width=2.8in]{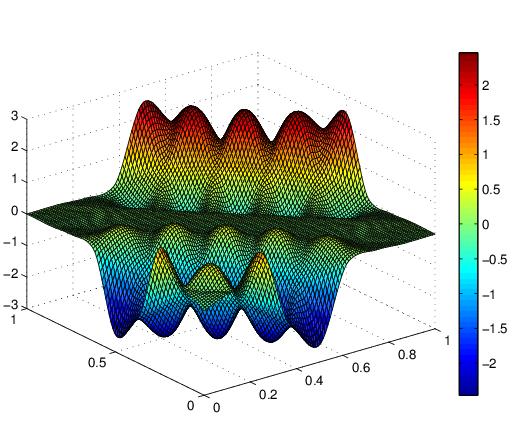}
\caption{The pressure $p_{h}^{n+1}$ of the original problem obtained by the fully discrete algorithm.}\label{figure_p7}
\end{figure}

Figure \ref{figure_p8} shows the analytical solution of pressure $p$ at $T=1$, Figure \ref{figure_p6} shows the numerical solution of pressure $p_{h}^{n+1}$ at $T=1$, which shows that there is no ``locking" phenomenon and approximates to the true solution of pressure $p$ very well.  Figure \ref{figure_p7} shows the pressure $p_{h}^{n+1}$ solved by finite element method without multiphysics reformulation for of the original problem, which shows that there is numerical oscillation for the pressure $p$.

\section{Conclusion}\label{sec-5}

In this paper, we reformulate the original problem to a new coupled fluid-fluid system by using variable substitution, that is, a generalized nonlinear Stokes problem of displacement vector field related to pseudo pressure and a diffusion problem of other pseudo pressure fields. And we analyze the existence and uniqueness of the reformulated problem. A fully discrete nonlinear finite element method is proposed to solve the model numerically, that is, a multiphysics finite element method is used for the space discretization and the backward Euler method is used as the time-stepping method, the Newton method is used to solve the generalized nonlinear Stokes problem. Then, Stability analysis and the optimal order error estimates are given. Numerical tests are given to demonstrate the theoretical results, which shows that our method has no numerical oscillation. To the best of our knowledge, it is the first time to propose a fully discrete multiphysics finite element method for the nonlinear poroelasticity model with nonlinear stress-strain relations and give the optimal error estimates.


\end{document}